\documentclass[reqno,12pt]{article}

\usepackage{amsmath,amssymb,amsthm}
\usepackage[UKenglish]{babel}
\usepackage{bbm} 
\usepackage[top=3.5cm,right=2.54cm,left=2.54cm]{geometry}
\usepackage[pdftitle={Quantum random walks and thermalisation II},
            pdfauthor={Alexander C. R. Belton},
            bookmarks=true,
            colorlinks=false,
            pdfstartview=FitH]{hyperref}
\usepackage{ifthen}
\usepackage{sectsty}

\sectionfont{\normalsize}
\subsectionfont{\normalsize}

\mathchardef\ordinarycolon\mathcode`\:
\mathcode`\:=\string"8000
\def\vcentcolon{\mathrel{\mathop\ordinarycolon}}
\begingroup\catcode`\:=\active\lowercase{\endgroup\let :\vcentcolon}

\theoremstyle{plain}
\newtheorem{theorem}{Theorem}[section]

\newtheorem{lemma}[theorem]{Lemma}
\newtheorem{proposition}[theorem]{Proposition}

\theoremstyle{definition}
\newtheorem{definition}[theorem]{Definition}
\newtheorem{example}[theorem]{Example}
\newtheorem{notation}[theorem]{Notation}
\newtheorem{remark}[theorem]{Remark}

\let\origthebibliography=\thebibliography
\def\thebibliography{\renewcommand{\section}[2]{}\origthebibliography}

\newcommand{\bebe}{\Gamma}
\newcommand{\comp}{\mathbin{\circ}}
\newcommand{\diag}{\delta}

\newcommand{\dialg}{\mathcal{D}}
\newcommand{\eevec}{\varepsilon}
\newcommand{\elltwo}{L^2( \R_+; \mul )}
\newcommand{\evec}[1]{\eevec(#1)}
\newcommand{\evecs}{\mathcal{E}}
\newcommand{\fock}{\mathcal{F}}
\newcommand{\GNS}{\bigl(\mkern2mu\mmul, \pi, \vac \bigr)}
\newcommand{\hilb}{\mathsf{H}}
\newcommand{\hilbb}{{\hilb\rb}}
\newcommand{\hlf}{\mbox{$\frac12$}}

\newcommand{\ini}{\mathsf{h}}
\newcommand{\Lind}{\mathcal{L}}
\newcommand{\mmul}{{\widehat{\mul}}}
\newcommand{\mmulb}{{\mmul\rb}}
\newcommand{\mul}{\mathsf{k}}
\newcommand{\opsp}{\mathsf{V}}
\newcommand{\opsq}{\mathsf{W}}
\newcommand{\ppi}{\wt{\pi}}
\newcommand{\sstate}{\wt{\state}}
\newcommand{\State}{\varrho}
\newcommand{\state}{\rho}
\newcommand{\statesp}{\mathsf{K}}
\newcommand{\vac}{\omega}
\newcommand{\vna}{\mathsf{A}}

\newcommand{\algten}{\mathbin{\odot}}
\newcommand{\bop}[2]%
{\ifthenelse{\equal{#2}{}}{\bopp( #1 )}{\bopp( #1; #2 )}}
\newcommand{\bopp}{\mathsf{B}}
\newcommand{\cb}{\mathrm{cb}}
\newcommand{\conj}[1]{#1^\dagger}
\newcommand{\dyad}[2]{\mathop{| #1 \rangle\langle #2 |}}
\newcommand{\hso}[1]{\bopp_2( #1 )}
\newcommand{\id}{I}
\newcommand{\im}{\mathop{\mathrm{im}}}
\newcommand{\indf}[1]{1_{#1}}
\newcommand{\kbo}[3]{#1\bopp( #2; #3 )}
\newcommand{\lin}{\mathop{\mathrm{lin}}}
\newcommand{\matten}{\mathbin{\otimes_{\mathrm{M}}}}
\newcommand{\mmodf}[2]{\nnmodf{}{#1}{#2}}
\newcommand{\nmodf}[2]{\nnmodf{\diag}{#1}{#2}}
\newcommand{\nnmodf}[3]{m_{#1}( #2, #3 )}
\newcommand{\rank}{\mathop{\mathrm{rank}}}
\newcommand{\rb}{\mathrm{b}}
\newcommand{\rd}{\mathrm{d}}
\newcommand{\std}{\,\rd}
\newcommand{\tfn}[1]{\mathbbm{1}_{#1}}
\newcommand{\tr}{\mathop{\mathrm{tr}}}
\newcommand{\tto}[1]{\rightarrow_{#1}}
\newcommand{\ttocb}{\tto{\cb}}
\newcommand{\ttommb}{\tto{\mkern2mu\mmulb}}
\newcommand{\uwkten}{\mathbin{\bar{\otimes}}}
\newcommand{\wt}[1]{\widetilde{#1}}

\newcommand{\etc}{\textit{et cetera}}
\newcommand{\ie}{\textit{i.e., }}
\renewcommand{\ge}{\geqslant}
\renewcommand{\le}{\leqslant}

\newcommand{\C}{\mathbb{C}}
\newcommand{\I}{\mathrm{i}}
\newcommand{\R}{\mathbb{R}}
\newcommand{\Z}{\mathbb{Z}}

\newcommand{\ds}{\displaystyle}
\newcommand{\gap}{\vspace{1ex}\noindent}
\newenvironment{mylist}%
{\begin{list}{}%
{\leftmargin 4.75em\labelwidth 3.5em\rightmargin 0.8em%
\topsep 0ex\itemsep 0.25ex}}%
{\end{list}}
{\begin{list}{}{}}{\end{list}}
\newcommand{\tu}[1]{\textup{#1}}

\allowdisplaybreaks[1]
\numberwithin{equation}{section}

\begin{document}

\begin{center}
{\large Quantum random walks and thermalisation II}\\[1ex]
{\normalsize Alexander C.~R. Belton}\\[0.5ex]
{\small Department of Mathematics and Statistics\\
Lancaster University, United Kingdom\\[0.5ex]
\href{mailto:a.belton@lancaster.ac.uk}%
{\textsf{a.belton@lancaster.ac.uk}} \qquad \today}
\end{center}

\begin{abstract}\noindent
A convergence theorem is obtained for quantum random walks with
particles in an arbitrary normal state. This result unifies and
extends previous work on repeated-interactions models, including that
of the author (2010, \textit{J.~London Math.\ Soc.}~(2)~\textbf{81},
412--434; 2010, \textit{Comm.\ Math.\ Phys.}~\textbf{300},
317--329). When the random-walk generator acts by ampliation and
multiplication or conjugation by a unitary operator, necessary and
sufficient conditions are given for the quantum stochastic cocycle
which arises in the limit to be driven by an isometric, co-isometric
or unitary process.
\end{abstract}

\gap

\gap

\section{Introduction}

The repeated-interactions framework, also called the theory of quantum
random walks or non-commutative Markov chains, has attracted much
attention. Physically, it describes a small quantum-mechanical system
interacting with a heat bath which is modelled by a chain of identical
particles. There have been many applications of this model; for
example, to quantum optics \cite{BJM06,BrP09,Gou04}, to quantum
control \cite{BHJ09,GoJ09} and to the dilation of quantum dynamical
semigroups \cite{Sah08}; for the latter, see also
\cite[Section~6]{Blt10a}. There are interesting connexions between
non-commutative Markov chains and multivariate operator theory
\cite{Goh10}.

Many results in this area (for example, those contained in
\cite{AtJ07b}, \cite{BJM06} and \cite{BrP09}) focus only on the
reduced dynamics, \ie the expectation semigroup which arises in the
limit. In contrast, the results obtained below provide a full
quantum-stochastic description of the limit dynamics. They may be
considered to be quantum analogues of Donsker's theorem, which gives
the convergence of suitably scaled classical random walks to Brownian
motion.

In previous work, the particles of the model were required to be
either in a vector state \cite{AtP06,Blt10a} or in a faithful normal
state~\cite{Blt10b}. Here a generalisation is obtained,
Theorem~\ref{thm:main}, which applies to quantum random walks with
particles in an arbitrary normal state; the previous results appear as
special cases.

Let $\state$ be a normal state on the particle algebra
$\bop{\statesp}{}$ and suppose that the linear map
$\Phi : \bop{\ini}{} \to \bop{\ini \otimes \statesp}{}$, which
describes the interaction between the system and a particle, depends
on the step-size parameter $\tau$. (For simplicity, the domain of the
generator is taken to be $\bop{\ini}{}$ throughout this introduction;
below it may be a general concrete operator space, or a von~Neumann
algebra for the applications in Section~\ref{sec:applications}.) In
order that the random walk with generator~$\Phi$ and particle
state~$\state$ converges to a limit cocycle, the mapping $\Phi$ must
behave correctly as the step size~$\tau \to 0$.

As shown in \cite{Blt10a}, when $\state$ is a vector state given by
$\vac \in \statesp$ then it is required that
\[
\begin{bmatrix}
\ds \frac{\Phi( a )^0_0 - a}{\tau} & %
\ds \frac{\Phi( a )^0_\times}{\sqrt{\tau}} \\[1ex]
\ds \frac{\Phi( a )^\times_0}{\sqrt{\tau}} & %
\ds \Phi( a )^\times_\times - a \otimes \id_\mul
\end{bmatrix} \to %
\begin{bmatrix}
\ds \Psi( a )^0_0 & \ds \Psi( a )^0_\times \\[1ex]
\ds \Psi( a )^\times_0 & \Psi( a )^\times_\times
\end{bmatrix}
\qquad \text{as } \tau \to 0 %
\qquad \text{for all } a \in \bop{\ini}{},
\]
where the convergence holds in a suitable topology and the matrix
decomposition
\[
\bop{\ini \otimes \statesp}{} \ni T = %
\begin{bmatrix}
 T_0^0 & T_\times^0 \\[1ex]
 T_0^\times & T_\times^\times
\end{bmatrix} \in %
\begin{bmatrix}
 \bop{\ini}{} & \bop{\ini \otimes \mul}{\ini} \\[1ex]
 \bop{\ini}{\ini \otimes \mul} & \bop{\ini \otimes \mul}{}
\end{bmatrix}
\]
corresponds to the Hilbert-space decomposition
$\statesp = \C\vac \oplus \mul$.

For the other extreme, where the normal state $\state$ is faithful, a
conditional expectation~$d$ on~$\bop{\statesp}{}$ which
preserves~$\state$ is required, and then
\[
( \tau^{-1} \diag + \tau^{-1 / 2} \diag^\perp ) %
( \Phi( a ) - a \otimes \id_\statesp ) = %
\frac{\diag( \Phi( a ) - a \otimes \id_\statesp )}{\tau} + %
\frac{\diag^\perp\bigl( \Phi( a ) \bigr)}{\sqrt{\tau}}
\]
must converge to $\Psi( a ) $ as $\tau \to 0$, where
$\diag := \id_{\bop{\ini}{}} \uwkten d$ and
$\diag^\perp = \id_{\bop{\ini \otimes \statesp}{}} - \diag$;
see~\cite{Blt10b}.

The general case is resolved below. Let $\State$ be the density matrix
that corresponds to the normal state $\state$, decompose $\statesp$ by
letting $\statesp_0 := ( \ker \State )^\perp$, and let~$d_0$ be a
conditional expectation on $\bop{\statesp_0}{}$ which preserves the
faithful state
\[
\state_0 : \bop{\statesp_0}{} \to \C; \ %
Z \mapsto \state\Bigl( \begin{bmatrix}
 Z & 0 \\[1ex]
 0 & 0
\end{bmatrix} \Bigr).
\]
The direct sum $\statesp = \statesp_0 \oplus \statesp_0^\perp$
provides a matrix decomposition of operators
in~$\bop{\ini \otimes \statesp}{}$ and the appropriate modification of
$\Phi( a )$ has the form
\begin{equation}\label{eqn:mod}
\begin{bmatrix}
\ds \frac{\diag_0( \Phi( a )_0^0 - %
a \otimes \id_{\statesp_0} )}{\tau} + %
\frac{\diag_0^\perp( \Phi( a )_0^0 )}{\sqrt{\tau}} & %
\ds \frac{\Phi( a )_\times^0}{\sqrt{\tau}} \\[1ex]
\ds \frac{\Phi( a )_0^\times}{\sqrt{\tau}} & %
\ds \Phi( a )_\times^\times - a \otimes \id_{\statesp_0^\perp}
\end{bmatrix},
\end{equation}
where $\diag_0 := \id_{\bop{\ini}{}} \uwkten d_0$ and similarly for
$d^\perp$. The top-left corner, where $\rho$ is faithful, is scaled by
$\tau^{-1}$ on the range of $\diag_0$ and by $\tau^{-1 / 2}$ off it;
elsewhere, the scaling is as for a vector state, with $\C \vac$
and~$\mul$ replaced by $\statesp_0$ and $\statesp_0^\perp$,
respectively.

A concrete realisation $\GNS$ of the GNS representation for $\state$
is employed to obtain the main result, Theorem~\ref{thm:main}; this
circumvents problems which arise from taking a quotient, when the
state is not faithful, in the standard approach. Let $\diag$ be the
conditional expectation on $\bop{\ini \otimes \statesp}{}$ obtained by
extending~$d_0$ and ampliating, so that
\[
\diag( a \otimes X ) = a \otimes %
\begin{bmatrix}
d_0( X_0^0 ) & 0 \\[1ex]
 0 & 0 \\
\end{bmatrix} %
\qquad \text{for all } a \in \bop{\ini}{} \text{ and } %
X \in \bop{\statesp}{};
\]
further, let $\ppi := \id_{\bop{\ini}{}} \uwkten \pi$ and
$\sstate := \id_{\bop{\ini}{}} \uwkten \state$. If the modification
\eqref{eqn:mod} converges to a limit~$\Psi$ in a suitable manner
then the embedded random walk with generator $\ppi \comp \Phi$
converges to a limit cocycle $j^\psi$ with generator
\begin{equation}\label{eqn:psimat}
\psi : a \mapsto \begin{bmatrix}
( \sstate \comp \Psi)( a ) & %
( \ppi \comp \diag^\perp \comp \Psi )( a )_\times^0 \\[2ex]
( \ppi \comp \diag^\perp \comp \Psi )( a )_0^\times & %
\ppi( \wt{P}_\times \Psi( a ) \wt{P}_\times )_\times^\times
\end{bmatrix},
\end{equation}
where the matrix decomposition here is that induced by writing
$\mmul$ as $\C\vac \oplus \mul$ and $\wt{P}_\times$ is the
orthogonal projection from $\ini \otimes \statesp$ onto
$\ini \otimes \statesp_0^\perp$.

The presence of the conditional expectation $\diag$ and the orthogonal
projection $\wt{P}_\times$ in the formula~\eqref{eqn:psimat} implies
that, in general, the number of independent noises in the quantum
stochastic differential equation satisfied by the cocycle $j^\psi$ is
fewer than might be expected. This thermalisation phenomenon, which
was first described in \cite{AtJ07a}, is quantified for particles with
finite degrees of freedom in Proposition~\ref{prp:indnoise}.

As is well known, if the cocycle generator $\psi$ acts by right
multiplication, \ie has the form
\[
\psi : a \mapsto ( a \otimes \id_\mmul ) G
\]
for some $G \in \bop{\ini \otimes \mmul}{}$, then
$j^\psi_t( a ) = ( a \otimes \id_\mmul ) X_t$ for all $t \ge 0$ and
$a \in \bop{\ini}{}$, and the driving process
$\bigl( X_t := j_t^\psi( \id_\ini ) \bigr)_{t \ge 0}$ is isometric or
co-isometric if and only if
\[
G + G^* + G^* \Delta G = 0 \qquad \text{or} \qquad
G + G^* + G \Delta G^* = 0,
\]
respectively, where $\Delta$ is the orthogonal projection from
$\ini \otimes \mmul$ onto $\ini \otimes \mul$. If $\Psi$ acts by right
multiplication then so does the map $\psi$ given
by~\eqref{eqn:psimat}, and Theorem~\ref{thm:hp} provides necessary and
sufficient conditions on $\Psi$ for the process which drives~$j^\psi$
to be isometric or co-isometric. This is used in
Theorems~\ref{thm:exhp} and~\ref{thm:eh} to show that random-walk
generators of the form
\[
a \mapsto ( a \otimes \id_\statesp ) %
\exp\bigl( -\I \tau H( \tau ) \bigr) \qquad \text{and} \qquad %
a \mapsto \exp\bigl( \I \tau H( \tau ) \bigr) %
( a \otimes \id_\statesp ) \exp\bigl( -\I \tau H( \tau ) \bigr),
\]
where the Hamiltonian $H( \tau )$ behaves correctly as $\tau \to 0$,
give rise to limit cocycles which are driven by unitary processes, \ie
they are of the form
\[
a \mapsto ( a \otimes \id_\mmul ) U_t \qquad \text{and} \qquad %
a \mapsto U_t^* ( a \otimes \id_\mmul ) U_t \qquad %
\text{for all } t \ge 0,
\]
where the process $( U_t )_{t \ge 0}$ is composed of unitary
operators.

This article is organised as follows. The basics of quantum random
walks on operator spaces are reviewed in Section~\ref{sec:qrw};
Section~\ref{sec:concrete} contains the concrete GNS representation
and some subsidiary results. The main theorem is established in
Section~\ref{sec:main}, and the final section,
Section~\ref{sec:applications}, gives some applications of the general
theory.

\subsection{Conventions and notation}

For the most part, the conventions and notation of
\cite{Blt10a,Blt10b} are followed; some innovations have been
introduced in an attempt to increase clarity. Vector spaces have
complex scalar field; inner products are linear in the second
variable. An empty sum or product equals the appropriate additive or
multiplicative unit.

The indicator function of a set~$S$ is denoted by $\indf{S}$; the sets
of non-negative integers and non-negative real numbers are denoted by
$\Z_+ := \{ 0, 1, 2, \ldots \}$ and $\R_+ := [ 0, \infty )$. The
identity transformation on a vector space~$V$ is denoted by~$\id_V$,
the linear span of $A \subseteq V$ is denoted by $\lin A$ and the
image and kernel of a linear transformation~$T$ on~$V$ are denoted
by~$\im T$ and $\ker T$; the sets of $m \times n$ and $n \times n$
matrices with entries in~$V$ are denoted by $M_{m, n}( V )$ and
$M_n( V )$. If the vectors $u$ and $v$ lie in an inner-product
space~$V$ then~$\dyad{u}{v}$ is the linear operator on $V$ such that
$w \mapsto \langle v, w \rangle u$; the orthogonal complement of
$A \subseteq V$ is denoted by $A^\perp$. Algebraic, Hilbert-space and
ultraweak tensor products are denoted by $\algten$, $\otimes$ and
$\uwkten$, respectively. The von~Neumann algebra of bounded operators
on a Hilbert space~$\hilb$ is denoted by $\bop{\hilb}{}$, and
$\bop{\hilb_1}{\hilb_2}$ denotes the Banach space of bounded operators
from Hilbert space~$\hilb_1$ to Hilbert space~$\hilb_2$.

\section{Walks with particles in the vacuum state}\label{sec:qrw}

\subsection{Toy and Boson Fock space}

\begin{definition}
Let $\mmul$ be a Hilbert space containing the distinguished unit
vector $\vac$ and let $\mul := \mmul \ominus \C \vac$ be the
orthogonal complement of $\C \vac$ in $\mmul$. Given $x \in \mul$, let
$\widehat{x} := \vac + x \in \mmul$.

The \emph{toy Fock space over $\mul$} is
$\bebe := \bigotimes_{n = 0}^\infty \mmul_{(n)}$, where
$\mmul_{(n)} := \mmul$ for all $n \in \Z_+$, with respect to the
stabilising sequence $(\vac_{(n)} := \vac)_{n = 0}^\infty$; the suffix
$(n)$ is used to indicate the relevant copy of~$\mmul$. Note that
$\bebe = \bebe_{n)} \otimes \bebe_{[n}$, where
$\bebe_{n)} := \bigotimes_{m = 0}^{n - 1}\mmul_{(m)}$ and
$\bebe_{[n} := \bigotimes_{m = n}^\infty\mmul_{(m)}$, for all
$n \in \Z_+$.
\end{definition}

\begin{notation}
Let $\fock = \fock_+\bigl( \elltwo \bigr)$ be the Boson Fock space
over $\elltwo$, the \mbox{Hilbert} space of square-integrable $\mul$-valued
functions on the half line. Recall that~$\fock$ may be considered as
the completion of $\evecs$, the linear span of
exponential vectors~$\evec{f}$ labelled by~$f \in \elltwo$,
with respect to the inner product
\[
\langle \evec{f}, \evec{g} \rangle := \exp\Bigl( %
\int_0^\infty \langle f( t ), g( t ) \rangle \std t %
\Bigr) \qquad \text{for all } f, g \in \elltwo.
\]
\end{notation}

\begin{proposition}
For all $\tau > 0$ there is a unique co-isometry
$D_\tau : \fock \to \bebe$ such that
\[
D_\tau \evec{f} = \bigotimes_{n = 0}^\infty \widehat{f( n ; \tau )}, %
\quad \text{where} \quad %
f( n ; \tau ) := %
\tau^{-1 / 2} \int_{n \tau}^{(n + 1) \tau} f( t )\std t,
\]
for all $f \in \elltwo$. Furthermore, $D_\tau^* D_\tau \to \id_\fock$
strongly as $\tau \to 0$.
\end{proposition}
\begin{proof}
See \cite[Section~2]{Blt08}.
\end{proof}

\subsection{Matrix spaces}

For more detail on the topics of this subsection and the next,
see~\cite{Lin05}.

Henceforth $\opsp$ is a fixed concrete operator space, \ie a
norm-closed subspace of~$\bop{\ini}{}$, where~$\ini$ is a Hilbert
space.

\begin{definition}
For a Hilbert space $\hilb$, the \emph{matrix space}
\[
\opsp \matten \bop{\hilb}{} := %
\{ T \in \bop{\ini \otimes \hilb}{} : %
E^x T E_y \in \opsp \text{ for all } x, y \in \hilb \}
\]
is an operator space, where
$E^x \in \bop{\ini \otimes \hilb}{\ini}$ is the adjoint of
\[
E_x : \ini \to \ini \otimes \hilb; \ u \mapsto u \otimes x.
\]
Note that
$\opsp \otimes \bop{\hilb}{} \subseteq \opsp \matten \bop{\hilb}{} %
\subseteq \opsp \uwkten \bop{\hilb}{}$, with the latter an equality if
$\opsp$ is ultraweakly closed, and
$\bigl( \opsp \matten \bop{\hilb_1}{} \bigr) %
\matten \bop{\hilb_2}{} = %
\opsp \matten \bop{\hilb_1 \otimes \hilb_2}{}$.
\end{definition}

\begin{definition}
If $\opsq$ is an operator space and $\hilb$ is a non-zero Hilbert
space then a linear map $\phi : \opsp \to \opsq$ is
\emph{$\hilb$~bounded} if $\| \phi \|_\hilbb < \infty$, where
\[
\| \phi \|_\hilbb := \left\{\begin{array}{ll}
( \dim \hilb ) \|\phi\| & \text{if } \dim \hilb < \infty, \\[1ex]
\| \phi \|_\cb & \text{if } \dim \hilb = \infty,
\end{array}\right.
\]
with $\| \cdot \|$ and $\| \cdot \|_\cb$ the operator and completely
bounded norms, respectively. The Banach space of all such
$\hilb$-bounded maps, with norm $\| \cdot \|_\hilbb$, is denoted by
$\kbo{\hilb}{\opsp}{\opsq}$.
\end{definition}

\begin{proposition}
Let $\Phi \in \kbo{\hilb}{\opsp}{\opsq}$.
The unique map
$\Phi \matten \id_{\bop{\hilb}{}} : %
\opsp \matten \bop{\hilb}{} \to \opsq \matten \bop{\hilb}{}$
such that
\[
E^x \bigl( \Phi \matten \id_{\bop{\hilb}{}}( T ) \bigr) E_y = %
\Phi( E^x T E_y ) %
\qquad \text{for all } x, y \in \hilb \text{ and } %
T \in \opsp \matten \bop{\hilb}{}
\]
is the \emph{$\hilb$ lifting of~$\Phi$}. This lifting is linear,
$\hilb$~bounded and such that
$\| \Phi \matten \id_{\bop{\hilb}{}} \| \le \| \Phi \|_\hilbb$; if
$\Phi$ is completely bounded then so is
$\Phi \matten \id_{\bop{\hilb}{}}$, with
$\| \Phi \matten \id_{\bop{\hilb}{}} \|_\cb \le \| \Phi \|_\cb$.
\end{proposition}
\begin{proof}
See \cite[Theorem~2.5]{Blt10a}.
\end{proof}

\begin{proposition}\label{prp:qrw}
Let
$\Phi \in %
\hilb\bopp\bigl( \opsp; \opsp \matten \bop{\hilb}{} \bigr)$.
There exists a unique family of maps
$\Phi^{(n)} : \opsp \to \opsp \matten \bop{\hilb^{\otimes n}}{}$
indexed by $n \in \Z_+$, the \emph{quantum random walk} with
\emph{generator} $\Phi$, such that $\Phi^{(0)} = \id_\opsp$ and
\[
E^x \Phi^{(n + 1)}( a ) E_y = \Phi^{(n)}( E^x \Phi( a ) E_y ) %
\qquad \text{for all } x, y \in \hilb, \, a \in \opsp %
\text{ and } n \in \Z_+.
\]
Each map is linear, $\hilb$~bounded and such that
$\| \Phi^{(n)} \|_\hilbb \le \| \Phi \|_\hilbb^n$ for all $n \ge 1$;
if $\Phi$ is completely bounded then so is $\Phi^{(n)}$, with
$\| \Phi^{(n)} \|_\cb \le \| \Phi \|_\cb^n$ for all $n \in \Z_+$.
\end{proposition}
\begin{proof}
See \cite[Theorem~2.7]{Blt10a}.
\end{proof}

\subsection{Quantum stochastic cocycles}

\begin{definition}
An \emph{$\ini$ process} $X$ is a family $( X_t )_{t \in \R_+}$ of
linear operators in $\ini \otimes \fock$, such that the domain of each
operator contains $\ini \algten \evecs$ and the map
$t \mapsto X_t u \evec{f}$ is weakly measurable for all $u \in \ini$
and $f \in \elltwo$; this process is \emph{adapted} if
\[
\langle u \evec{f}, X_t v \evec{g} \rangle = %
\langle u \evec{\indf{[ 0, t )} f}, %
X_t v \evec{\indf{[ 0, t )} g} \rangle %
\langle \evec{\indf{[ t, \infty )} f}, %
\evec{\indf{[ t, \infty )} g} \rangle
\]
for all $u$,~$v \in \ini$, $f$,~$g \in \elltwo$ and $t \in \R_+$. (As
is conventional, the tensor-product sign is omitted between elements
of $\ini$ and exponential vectors.)

A \emph{mapping process} $j$ is a family
$\bigl( j_\cdot( a ) \bigr)_{a \in \opsp}$ of $\ini$ processes such
that the map $a \mapsto j_t( a )$ is linear for all $t \in \R_+$; this
process is \emph{adapted} if each $\ini$~process $j_\cdot( a )$ is,
and is \emph{strongly regular} if
\[
j_t( \cdot ) E_{\evec{f}} \in %
\bopp\bigl( \opsp; \bop{\ini}{\ini \otimes \fock} \bigr) %
\qquad \text{for all } f \in \elltwo \text{ and } t \in \R_+,
\]
with norm locally uniformly bounded as a function of $t$.
\end{definition}

\begin{theorem}
Let
$\psi \in %
\mmul\bopp\bigl( \opsp; \opsp \matten \bop{\mmul}{} \bigr)$.
There exists a unique strongly regular adapted mapping process
$j^\psi$, the \emph{quantum stochastic cocycle generated by~$\psi$},
such that
\begin{equation}\label{eqn:qsdeip}
\langle u \evec{f}, ( j^\psi_t( a ) - %
a \otimes \id_\fock ) v \evec{g} \rangle = %
\int_0^t \langle u \evec{f}, j^\psi_s( E^{\widehat{f( s )}} \psi( a ) %
E_{\widehat{g( s )}}) v \evec{g} \rangle \std s
\end{equation}
for all $u$,~$v \in \ini$, $f$,~$g \in \elltwo$, $a \in \opsp$ and
$t \in \R_+$. The process $j^\psi$ has the Feller property, in the
sense that $E^{\evec{f}} j^\psi_t( a ) E_{\evec{g}} \in \opsp$ for all
$f$,~$g \in \elltwo$, $a \in \opsp$ and $t \in \R_+$. If $\psi$ is
completely bounded then so is $j^\psi_t( \cdot ) E_{\evec{f}}$, for
all $f \in \elltwo$ and $t \in \R_+$.
\end{theorem}
\begin{proof}
This is a result of Lindsay and Wills \cite{LiW01}.
\end{proof}

\begin{remark}
The fact that \eqref{eqn:qsde} holds is equivalent to saying that the
strongly regular adapted mapping process~$j^\psi$ satisfies the
quantum stochastic differential equation
\begin{equation}\label{eqn:qsde}
\rd j^\psi_t( a ) = j^\psi_t \std \Lambda_{\psi( a )}( t ) \qquad %
\text{for all } t \in \R_+,
\end{equation}
with the initial condition $j^\psi_0( a ) = a \otimes \id_\fock$, for
all $a \in \vna$.
\end{remark}

\begin{definition}
Let $\tau > 0$ and
$\Phi \in %
\mmul\bopp\bigl( \opsp; \opsp \matten \bop{\mmul}{} \bigr)$.
The \emph{embedded random walk} with \emph{generator}~$\Phi$ and
\emph{step size}~$\tau$ is the mapping process $J^{\Phi, \tau}$ such
that
\[
J^{\Phi, \tau}_t( a ) := (\id_\ini \otimes D_\tau)^* %
( \Phi^{(n)}( a ) \otimes \id_{\bebe_{[n}} ) %
( \id_\ini \otimes D_\tau ) \qquad %
\text{if } t \in [ n \tau, ( n + 1 ) \tau )
\]
for all $a \in \opsp$ and $n \in \Z_+$.
\end{definition}

\begin{notation}
Let $\tau > 0$ and
$\Phi \in \mmul\bopp\bigl( \opsp; \opsp \matten \bop{\mmul}{} \bigr)$,
and let $\Delta$ denote the orthogonal projection from
$\ini \otimes \mmul$ onto $\ini \otimes \mul$, with
$\Delta^\perp := \id_{\ini \otimes \mmul} - \Delta$. The modification
\begin{equation}\label{eqn:modif}
\mmodf{\Phi}{\tau} : \opsp \to \opsp \matten \bop{\mmul}{}; \ %
a \mapsto %
( \tau^{-1 / 2} \Delta^\perp + \Delta ) %
( \Phi( a ) - a \otimes \id_\mmul ) %
( \tau^{-1 / 2} \Delta^\perp + \Delta )
\end{equation}
is $\mmul$~bounded, and is completely bounded whenever $\Phi$ is.
\end{notation}

\begin{theorem}\label{thm:qrw}
Let $\tau_n > 0$ and
$\Phi_n$,~$\psi \in %
\mmul\bopp\bigl(\opsp; \opsp \otimes \bop{\mmul}{} \bigr)$
be such that
\[
\tau_n \to 0 \qquad \text{and} \qquad %
\mmodf{\Phi_n}{\tau_n} \matten \id_{\bop{\mmul}{}} \to %
\psi \matten \id_{\bop{\mmul}{}} \text{ strongly},
\]
\ie pointwise in norm, as $n \to \infty$. Then
\begin{equation}\label{eqn:conv}
\lim_{n \to \infty} \sup_{t \in [ 0, T ]} %
\| J^{\Phi_n, \tau_n}_t( a ) E_{\evec{f}} - %
j^\psi_t( a ) E_{\evec{f}} \| = 0 %
\quad \text{for all } a \in \opsp, \, f \in \elltwo %
\text{ and } T \in \R_+.
\end{equation}
If, further, $\| \mmodf{\Phi_n}{\tau_n} - \psi \|_\mmulb \to 0$ as
$n \to \infty$ then
\begin{equation}\label{eqn:sconv}
\lim_{n \to \infty} \sup_{t \in [ 0, T ]} %
\| J^{\Phi_n, \tau_n}_t( \cdot ) E_{\evec{f}} - %
j^\psi_t( \cdot ) E_{\evec{f}} \|_\mmulb = 0 \quad \text{for all } %
f \in \elltwo \text{ and } T \in \R_+;
\end{equation}
when $\Phi_n$ and $\psi$ are completely bounded, the same holds with
$\| \cdot \|_\mmulb$ replaced by~$\| \cdot \|_\cb$.
\end{theorem}
\begin{proof}
See \cite[Theorem~7.6]{Blt10a}.
\end{proof}

\begin{notation}
For brevity, the conclusion \eqref{eqn:conv} will be denoted by
$J^{\Phi, \tau}\to j^\psi$; the stronger conclusion \eqref{eqn:sconv}
will be denoted by $J^{\Phi, \tau} \ttommb j^\psi$, or by
$J^{\Phi, \tau} \ttocb j^\psi$ if the completely bounded version
holds.
\end{notation}

\begin{remark}
If $\hilb$ is infinite dimensional and
$\phi_n \in \kbo{\hilb}{\opsp}{\opsq}$ then
$\phi_n \matten \id_{\bop{\hilb}{}} \to 0$ strongly if and only if
$\| \phi_n \|_\cb \to 0$ \cite[Lemma~2.13]{Blt10a}.
\end{remark}

\section{A concrete GNS representation}\label{sec:concrete}

\begin{definition}
If $\hilb$ is a Hilbert space then $\conj{\hilb}$ denotes the Hilbert
space \emph{conjugate} to $\hilb$; thus
$\conj{\hilb} := \{ \conj{u} : u \in \hilb \}$, with
\[
\conj{u} + \conj{v} := \conj{( u + v )}, \quad
\lambda \conj{u} := \conj{( \overline{\lambda} u )} \quad
\text{and} \quad %
\langle \conj{u}, \conj{v} \rangle := %
\langle v, u \rangle
\]
for all $u$,~$v \in \hilb$ and $\lambda \in \C$. Note that the map
\[
\conj{\phantom{M}} : \bop{\hilb}{} \to \bop{\conj{\hilb}}{}; \ %
\conj{T} ( \conj{u} ) := \conj{( T u )} %
\qquad \text{for all } T \in \bop{\hilb}{} \text{ and } u \in \hilb
\]
is anti-linear and isometric, and it commutes with the adjoint.
\end{definition}

\begin{notation}
If $\hilb$ is a Hilbert space then $\hso{\hilb}$ is the Hilbert space
of Hilbert--Schmidt operators on $\hilb$, with inner product
$\langle S, T \rangle := \tr( S^* T )$ where $\tr$ is the standard
trace on $\bop{\hilb}{}$. Recall that $\hso{\hilb}{}$ is a two-sided
$*$-ideal in the $*$-algebra $\bop{\hilb}{}$.
\end{notation}

\begin{proposition}\label{prp:iso}
The isometric isomorphism
$U_\hilb : \hso{\hilb} \to \hilb \otimes \conj{\hilb}$ determined by
the requirement that
$U_\hilb( \dyad{u}{v} ) = u \otimes \conj{v}$ for all $u$,
$v \in \hilb$ is such that
\begin{equation}
U_\hilb( X T Y^* ) = ( X \otimes \conj{Y} ) U_\hilb( T ) %
\qquad \text{for all } T \in \hso{\hilb} %
\text{ and } X, Y \in \bop{\hilb}{}.
\end{equation}
\end{proposition}
\begin{proof}
This is elementary.
\end{proof}

\begin{notation}
Let $\state$ be a normal state on $\bop{\statesp}{}$ with density
matrix $\State \in \bop{\statesp}{}$, so that
\[
\State \ge 0, \quad \State^{1 / 2} \in \hso{\statesp}, \quad %
\| \State^{1 / 2} \|_2 = 1 \quad \text{and} \quad %
\state( X ) = \tr( \State X ) %
\qquad \text{for all } X \in \bop{\statesp}{}.
\]
Let $P_0$ denote the orthogonal projection from $\statesp$ onto
$\statesp_0 := %
\overline{\im \State^{1 / 2}} = ( \ker \State^{1 / 2} )^\perp$, where
$\overline{\vphantom{M}\,\cdot\,}$ denotes norm closure.
\end{notation}

\begin{proposition}\label{prp:concrete}
Let $\mmul := \statesp \otimes \conj{\statesp_0}$. The injective
normal unital $*$-homomorphism
\[
\pi : \bop{\statesp}{} \to \bop{\mmul}{}; \ %
X \mapsto X \otimes \id_{\conj{\statesp_0}},
\]
the \emph{concrete GNS representation}, has cyclic vector
$\vac := U_\statesp ( \State^{1 / 2} ) \in \mmul$ such that
\[
\langle \vac, \pi( X ) \vac \rangle_\mmul = \state( X ) %
\qquad \text{for all } X \in \bop{\statesp}{}
\]
and
\begin{equation}\label{eqn:stateker}
\state( X P_0 ) = \state( X ) = \state( P_0 X ) %
\qquad \text{for all } X \in \bop{\statesp}{}.
\end{equation}
\end{proposition}
\begin{proof}
Note that $\State^{1 / 2} = \State^{1 / 2} P_0$, since
$\statesp_0^\perp = \ker \State^{1/2}$, and so, by
Proposition~\ref{prp:iso},
\[
\vac = U_\statesp ( \State^{1 / 2} P_0 ) = %
( \id_\statesp \otimes \conj{P_0} ) U_\statesp %
( \State^{1 / 2} ) \in \statesp \otimes \conj{\statesp_0} = \mmul
\]
and
\[
\langle \vac, \pi( X ) \vac \rangle_\mmul = %
\langle U_\statesp ( \State^{1 / 2} ), %
U_\statesp ( X \State^{1 / 2} ) \rangle_\mmul = %
\tr( \State X ) = \state( X ) %
\qquad \text{for all } X \in \bop{\statesp}{}.
\]
By Proposition~\ref{prp:iso},
\[
\pi( \dyad{u}{v} ) \vac = %
U_\statesp ( \dyad{u}{v} \State^{1 / 2} ) = %
u \otimes \conj{( \State^{1 / 2} v )} \qquad %
\text{for all } u, v \in \statesp,
\]
thus
\[
\{ \pi( X ) \vac : X \in \bop{\statesp}{} \} \supseteq %
\lin \{ u \otimes \conj{( \State^{1 / 2} v )} : %
u, v \in \statesp \} = %
\statesp \algten \conj{(\im \State^{1 / 2})}
\]
and $\vac$ is cyclic for $\pi$. Finally,
\[
\state( P_0 X ) = \tr( \State P_0 X ) = \tr( \State X ) = %
\state( X ) \qquad \text{for all } X \in \bop{\statesp}{}
\]
and, similarly, $\state( X P_0 ) = \state( X )$.
\end{proof}

\begin{notation}
For brevity, let
\[
[ X ] := \pi( X ) \vac = U_\statesp( X \State^{1 / 2} )
\qquad \text{for all } X \in \bop{\statesp}{},
\]
where $U_\statesp$ is as in Proposition~\ref{prp:iso}. Note that
$[ X ] \in \mul := ( \C \vac )^\perp$ if and only if
$X \in \ker \state$.
\end{notation}

\begin{proposition} 
The ampliated representation
\[
\ppi := \id_{\bop{\ini}{}} \uwkten \pi : %
\bop{\ini \otimes \statesp}{} \to \bop{\ini \otimes \mmul}{}; \ %
T \mapsto T \otimes \conj{\id_0}
\]
is an injective normal unital $*$-homomorphism such that
$\ppi\bigl( \opsp \matten \bop{\statesp}{} \bigr) \subseteq %
\opsp \matten \bop{\mmul}{}$. The slice map
\[
\sstate := \id_{\bop{\ini}{}} \uwkten \state : %
\bop{\ini \otimes \statesp}{} \to \bop{\ini}{}
\]
is completely positive, normal and such that
\begin{equation}\label{eqn:homstate}
E^{[ X ]} \ppi( T ) E_{[ Y ]} = %
\sstate\bigl( ( \id_\ini \otimes X )^* T ( \id_\ini \otimes Y ) \bigr)
\qquad \text{for all } T \in \bop{\ini \otimes \statesp}{} %
\text{ and } X, Y \in \bop{\statesp}{};
\end{equation}
in particular,
$\sstate\bigl( \opsp \matten \bop{\statesp}{} \bigr) \subseteq \opsp$.
Furthermore,
\begin{equation}\label{eqn:sstateker}
\sstate\bigl( \wt{P}_0 T \bigr) = \sstate( T ) = %
\sstate\bigl( T \wt{P}_0 \bigr) %
\qquad \text{for all } T \in \bop{\ini \otimes \statesp}{},
\end{equation}
where $\wt{P}_0 := \id_\ini \otimes P_0$ is the orthogonal projection
from $\ini \otimes \statesp$ onto $\ini \otimes \statesp_0$.
\end{proposition}
\begin{proof}
The existence of $\ppi$ and $\sstate$ is standard; see
\cite[Theorem~IV.5.2 \& Proposition~IV.5.13]{Tak79}. Furthermore,
\[
\ppi\bigl( \opsp \matten \bop{\statesp}{} \bigr) \subseteq %
\bigl( \opsp \matten \bop{\statesp}{} \bigr) \algten %
\bop{\conj{\statesp_0}}{} \subseteq %
\opsp \matten \bop{\statesp \otimes \conj{\statesp_0}}{} = %
\opsp \matten \bop{\mmul}{}.
\]
Next, observe that if $a \in \bop{\ini}{}$ and
$X \in \bop{\statesp}{}$ then
\[
E^\vac \ppi( a \otimes X ) E_\vac = %
\langle \vac, \pi( X ) \vac \rangle_\mmul \, a = %
\state( X ) a = \sstate( a \otimes X );
\]
the identity \eqref{eqn:homstate} follows from this, continuity and
the fact that
\[
E_{[ Y ]} = \bigl( \id_\ini \otimes \pi( Y ) \bigr) E_\vac = %
\ppi( \id_\ini \otimes Y ) E_\vac %
\qquad \text{for all } Y \in \bop{\statesp}{}.
\]
The final claim is an immediate consequence of \eqref{eqn:stateker}.
\end{proof}

\begin{notation}
Let $F_0 : \statesp_0 \hookrightarrow \statesp$ be the canonical
embedding, so that $F_0 F_0^* = P_0$, the orthogonal projection from
$\statesp$ onto $\statesp_0$, and $F_0^* F_0 = \id_{\statesp_0}$, the
identity map on $\statesp_0$. The direct-sum decomposition
$\statesp = \statesp_0 \oplus \statesp_0^\perp$ will be used to write
operators as two-by-two matrices.
\end{notation}

\begin{lemma}\label{lem:faithfulstate}
The map
\[
\state_0 : \bop{\statesp_0}{} \to \C; \ %
X \mapsto \state( F_0 X F_0^* ) = %
\state\bigl( \left[\begin{smallmatrix}
 X & 0 \\[0.5ex]
 0 & 0
\end{smallmatrix}\right] \bigr) = %
\tr( F_0^* \State F_0 X )
\]
is a faithful normal state.
\end{lemma}
\begin{proof}
That $\state_0$ is a normal positive linear functional is
immediate. As $\State$ is compact and positive, there exists an
orthonormal set $\{ e_j : j \in J \} \subseteq \statesp$ such that
\[
\State = \sum_{j \in J} \lambda_j \dyad{e_j}{e_j},
\]
where $\lambda_j > 0$ for all $j \in J$ and
$\sum_{j \in J} \lambda_j = 1$. Since $e_j \in \statesp_0$ for all
$j \in J$, it follows that
\[
\state_0( X ) = \tr( \State F_0 X F_0^* ) = %
\sum_{j \in J} \lambda_j %
\langle e_j, F_0 X F_0^* e_j \rangle = %
\sum_{j \in J} \lambda_j \langle e_j, X e_j \rangle %
\qquad \text{for all } X \in \bop{\statesp_0}{};
\]
in particular, $\state_0( \id_{\statesp_0} ) = 1$. Furthermore,
$\{ e_j : j \in J \}$ is total in $\statesp_0$, so if
$X \in \bop{\statesp_0}{}$ then
\[
\state_0( X^* X ) = 0 \iff %
\sum_{j \in J} \lambda_j \| X e_j \|^2 = 0 \iff X = 0.
\qedhere
\]
\end{proof}

\begin{notation}
Fix a conditional expectation $d_0$ from $\bop{\statesp_0}{}$ onto
a $*$-subalgebra $\dialg_0$, and suppose $d_0$ preserves the state
$\state_0$.

By definition,~$d_0$ is a completely positive linear idempotent which
is $\dialg_0$~linear, \ie a module map for the natural
$\dialg_0-\dialg_0$-bimodule structure on $\bop{\statesp_0}{}$.
As $\state_0 \comp d_0 = \state_0$ and $\state_0$ is faithful, it
follows that $d_0$ is ultraweakly continuous, so $\dialg_0$ is a
von~Neumann algebra, and $d_0$ is unital, \ie
$d_0( \id_{\statesp_0} ) = \id_{\statesp_0}$.
\end{notation}

\begin{proposition}\label{prp:condexp}
The ultraweakly continuous map
\[
d : \bop{\statesp}{} \to \bop{\statesp}{}; \ %
X = %
\left[\begin{smallmatrix}
 X_0^0 & X_\times^0 \\[0.5ex]
 X_0^\times & X_\times^\times
\end{smallmatrix}\right] \mapsto %
F_0 d_0( F_0^* X F_0 ) F_0^* = %
\left[\begin{smallmatrix}
 d_0( X_0^0 ) \ & 0 \vphantom{X_0^0} \\[1ex]
 0 \ & 0 \vphantom{X_0^0}
\end{smallmatrix}\right]
\]
is a conditional expectation onto $F_0 \dialg_0 F_0^*$ such that
$\state \comp d = \state$, so the ultraweakly continuous ampliation
\[
\diag := \id_{\bop{\ini}{}} \uwkten d : %
\bop{\ini \otimes \statesp}{} \to \bop{\ini \otimes \statesp}{}
\]
is a conditional expectation onto
$\bop{\ini}{} \uwkten F_0 \dialg_0 F_0^*$ such that
$\sstate \comp \diag = \sstate$ and
\begin{equation}\label{eqn:diag}
\diag\bigl( \opsp \matten \bop{\statesp}{} \bigr) \subseteq %
\opsp \matten \bop{\statesp}{}.
\end{equation}
\end{proposition}
\begin{proof}{}
The maps $d$ and $\diag$ inherit linearity, idempotency, complete
positivity and ultraweak continuity from~$d_0$; furthermore,
\begin{alignat*}{2}
d\bigl( d( X ) Y \bigr) & = %
F_0 d_0( F_0^* F_0 d_0( F_0^* X F_0 ) F_0^* Y F_0 ) F_0^* \\
 & = F_0 d_0( F_0^* X F_0 ) d_0( F_0^* Y F_0 ) F_0^* & = d( X ) d( Y ) %
\qquad \text{for all } X, Y \in \bop{\statesp}{}
\end{alignat*}
and, using the adjoint, $d\bigl( X d( Y ) \bigr) = d( X ) d( Y )$.
Thus $d$ and $\diag$ are conditional expectations.

To see that states are preserved, let $X \in \bop{\statesp}{}$ and
recall that $d_0$ preserves $\state_0$, so
\[
\state\bigl( d( X ) \bigr) = %
\state_0\bigl( d_0( F_0^* X F_0 ) \bigr) = %
\state_0( F_0^* X F_0 ) = \state( P_0 X P_0 ) = \state( X ),
\]
where the final equality follows from \eqref{eqn:stateker}.

Finally, let $T \in \bop{\ini \otimes \statesp}{}$ and note that
$T \in \opsp \matten \bop{\statesp}{}$ if and only if
$( \id_{\bop{\ini}{}} \uwkten \phi )( T ) \in \opsp$ for every
normal linear functional $\phi$ on $\bop{\statesp}{}$. As $d$
is ultraweakly continuous, the inclusion~\eqref{eqn:diag} follows.
\end{proof}

\section{Walks with an arbitrary normal particle state}%
\label{sec:main}

Throughout this section, $\state$ is a normal state on
$\bop{\statesp}{}$ corresponding to the density matrix $\State$, the
subspace $\statesp_0 = \overline{\im \State^{1 / 2} }$ and $\GNS$ is
the concrete GNS representation of Proposition~\ref{prp:concrete}. 

\begin{notation}
Given $\Phi \in %
\statesp\bopp\bigl( \opsp; \opsp \matten \bop{\statesp}{} \bigr)$, let
$\Phi'( a ) := \Phi( a ) - a \otimes \id_\statesp$ for all
$a \in \opsp$. Recall that $\wt{P}_0$ is the orthogonal projection
from $\ini \otimes \statesp$ onto $\ini \otimes \statesp_0$.
\end{notation}

The following definition gives the correct modification of a generator
for a quantum random walk with particle state $\state$ and conditional
expectation $d_0$.

\begin{definition}
Let $\tau > 0$ and
$\Phi \in %
\statesp\bopp\bigl( \opsp; \opsp \matten \bop{\statesp}{} \bigr)$.
The modification
\begin{align}\label{eqn:modf}
\nmodf{\Phi}{\tau} : \opsp \to \opsp \matten \bop{\statesp}{}; \ %
a \mapsto & \wt{P}_0 %
( \tau^{-1} \diag + \tau^{-1 / 2} \diag^\perp )%
\bigl( \Phi'( a ) \bigr) \wt{P}_0 \nonumber \\
 & + \tau^{-1 / 2} \wt{P}_0 \Phi( a ) \wt{P}_0^\perp + %
\tau^{-1 / 2} \wt{P}_0^\perp \Phi( a ) \wt{P}_0 + %
\wt{P}_0^\perp \Phi'( a ) \wt{P}_0^\perp
\end{align}
is $\statesp$ bounded, and is completely bounded whenever~$\Phi$ is.
\end{definition}

\begin{remark}
The modification \eqref{eqn:modf} acts as follows: on the block
corresponding to $\statesp_0 \times \statesp_0$, the scaling
r\'{e}gime appropriate for a faithful normal state is adopted
\cite{Blt10b}; on the blocks corresponding to
$\statesp_0 \times \statesp_0^\perp$, $\statesp_0^\perp \times
\statesp_0$ and $\statesp_0^\perp \times \statesp_0^\perp$, the
scaling is that used for the vector-state situation,
Theorem~\ref{thm:qrw}, with~$\statesp_0$ playing the r\^{o}le of
$\C \vac$ and $\statesp_0^\perp$ that of $\mul$.

In particular, if $\state$ is faithful then \eqref{eqn:modf} is the
same modification as in \cite[Definition~11]{Blt10b}, whereas if
$\state$ is a vector state then $d_0$ must be the identity map and
the modification is the same as that given in~\eqref{eqn:modif}.
\end{remark}

\begin{lemma}\label{lem:cruc}
Let $\tau > 0$ and
$\Phi \in %
\statesp\bopp\bigl( \opsp; \opsp \matten \bop{\statesp}{} \bigr)$.
Then
\[
( \tau \diag + \tau^{1 / 2} \diag^\perp )%
\bigl( \nmodf{\Phi}{\tau}( a ) \bigr) = %
\Phi'( a ) + %
( \tau^{1 / 2} - 1 ) \wt{P}_0^\perp \Phi'( a ) \wt{P}_0^\perp
\qquad \text{for all } a \in \opsp.
\]
\end{lemma}
\begin{proof}
Note first that, as $d_0( \id_{\statesp_0} ) = \id_{\statesp_0}$, it
follows that $d( \id_\statesp ) = P_0$ and so, by the bimodule
property for a conditional expectation,
\[
d( P_0 X ) = d( X ) = d( X P_0 ) %
\qquad \text{for all } X \in \bop{\statesp}{}.
\]
Hence, using the bimodule property again,
\begin{align}
\wt{P}_0 \diag( T ) & = \diag( \wt{P}_0 T ) = \diag( T ) = %
\diag( T \wt{P}_0 ) = \diag( T ) \wt{P}_0 \nonumber\\[1ex]
\text{and} \quad \wt{P}_0^\perp \diag( T ) & = %
\diag( \wt{P}_0^\perp T ) = 0 = %
\diag( T \wt{P}_0^\perp ) = \diag( T ) \wt{P}_0^\perp %
\qquad \text{for all } T \in \bop{\ini \otimes \statesp}{}.
\label{eqn:diagker}
\end{align}
Consequently,
\begin{align*}
( \tau \diag + \tau^{1 / 2} \diag^\perp )%
\bigl( \nmodf{\Phi}{\tau}( a ) \bigr) & = %
\diag\bigl( \Phi'( a ) \bigr) + %
\tau^{1 / 2} \nmodf{\Phi}{\tau}( a ) - %
\tau^{-1 / 2} \diag\bigl( \Phi'( a ) \bigr) \\[1ex]
& = ( 1 - \tau^{-1 / 2} ) \diag\bigl( \Phi'( a ) \bigr) + %
\wt{P}_0 ( \tau^{-1 / 2} \diag + \diag^{\perp} )%
\bigl( \Phi'( a ) \bigr) \wt{P}_0 \\
 & \quad + \wt{P}_0 \Phi( a ) \wt{P}_0^\perp + %
\wt{P}_0^\perp \Phi( a ) \wt{P}_0 + %
\tau^{1 / 2} \wt{P}_0^\perp \Phi'( a ) \wt{P}_0^\perp \\[1ex]
& = ( 1 - \tau^{-1 / 2} ) \diag\bigl( \Phi'( a ) \bigr) + %
( \tau^{-1 / 2} - 1 ) \wt{P}_0 \diag\bigl( \Phi'( a ) \bigr) %
\wt{P}_0 \\
 & \quad + \wt{P}_0 \Phi'( a ) \wt{P}_0 + %
\wt{P}_0 \Phi( a ) \wt{P}_0^\perp + %
\wt{P}_0^\perp \Phi( a ) \wt{P}_0 + %
\tau^{1 / 2} \wt{P}_0^\perp \Phi'( a ) \wt{P}_0^\perp \\[1ex]
& = \Phi'( a ) + ( \tau^{1 / 2} - 1 ) %
\wt{P}_0^\perp \Phi'( a ) \wt{P}_0^\perp.
\qedhere
\end{align*}
\end{proof}

The following theorem gives a convergence result for quantum random
walks with particles in the arbitrary normal state $\state$. Recall
that $\Delta$ denotes the orthogonal projection from
$\ini \otimes \mmul$ onto $\ini \otimes \mul$.

\begin{theorem}\label{thm:main}
Let $\tau_n > 0$ and
$\Phi_n$,~$\Psi \in %
\statesp\bopp\bigl( \opsp; \opsp \matten \bop{\statesp}{} \bigr)$
be such that
\[
\tau_n \to 0 \qquad \text{and} \qquad %
\nmodf{\Phi_n}{\tau_n} \matten \id_{\bop{\statesp}{}} \to %
\Psi \matten \id_{\bop{\statesp}{}} \text{ strongly} %
\qquad \text{as } n \to \infty.
\]
Define
$\psi \in %
\mmul\bopp\bigl( \opsp; \opsp \matten \bop{\mmul}{} \bigr)$ by setting
\begin{align}\label{eqn:psidef}
\psi( a ) & := %
\Delta^\perp ( \ppi \comp \Psi )( a ) \Delta^\perp + %
\Delta^\perp ( \ppi \comp \diag^\perp \comp \Psi )( a ) \Delta %
\nonumber \\
& \qquad + %
\Delta ( \ppi \comp \diag^\perp \comp \Psi )( a ) \Delta^\perp + %
\Delta \ppi ( \wt{P}_0^\perp \Psi( a ) \wt{P}_0^\perp ) \Delta %
\qquad \text{for all } a \in \opsp,
\end{align}
and note that $\psi$ is completely bounded if $\Psi$ is. Then
$J^{\ppi \comp \Phi, \tau} \to j^\psi$; furthermore,
\[
\text{if} \quad \| \nmodf{\Phi_n}{\tau_n} - \Psi \|_\statesp \to 0 %
\quad \text{then} \quad J^{\ppi \comp \Phi, \tau} \ttommb j^\psi
\]
and, when $\Phi_n$ and $\Psi$ are completely bounded,
\[
\text{if} \quad \| \nmodf{\Phi_n}{\tau_n} - \Psi \|_\cb \to 0 %
\quad \text{then} \quad J^{\ppi \comp \Phi, \tau} \ttocb j^\psi.
\]
\end{theorem}
\begin{proof}
Let $a \in \opsp$ and, for brevity, let $\tau = \tau_n$ and
$\Phi = \Phi_n$. Note first that
\[
E^\vac \mmodf{\ppi \comp \Phi}{\tau}( a ) E_\vac = %
\tau^{-1} E^\vac \ppi\bigl( \Phi'( a ) \bigr) E_\vac = %
\tau^{-1} \sstate\bigl( \Phi'( a ) \bigr),
\]
by \eqref{eqn:modif} and \eqref{eqn:homstate}, whereas
\[
E^\vac \ppi\bigl( \nmodf{\Phi}{\tau}( a ) \bigr) E_\vac = %
\sstate\bigl( \nmodf{\Phi}{\tau}( a ) \bigr) = %
\tau^{-1} \sstate\bigl( \Phi'( a ) \bigr),
\]
with the second equality a consequence of \eqref{eqn:sstateker} and
the fact that $\diag$ preserves $\sstate$. Hence
\[
E^\vac \mmodf{\ppi \comp \Phi}{\tau}( a ) E_\vac = %
E^\vac \ppi\bigl( \nmodf{\Phi}{\tau}( a ) \bigr) E_\vac.
\]
Next, let $X \in \ker \state$ and use \eqref{eqn:modif} and
\eqref{eqn:homstate} again to see that
\[
E^\vac \mmodf{\ppi \comp \Phi}{\tau}( a ) E_{[ X ]} = \tau^{-1 / 2} %
\sstate\bigl( \Phi( a ) ( \id_\ini \otimes X ) \bigr) = %
\sstate\bigl( \tau^{-1 / 2} \Phi'( a ) ( \id_\ini \otimes X ) \bigr),
\]
where the second equality holds because
$\sstate( a \otimes X ) = \state( X ) a = 0$. As $\diag$
preserves $\sstate$, so
\[
\sstate\bigl( \wt{P}_0^\perp \Phi'( a ) \wt{P}_0^\perp %
( \id_\ini \otimes X ) \bigr) = %
( \sstate \comp \diag )\bigl( \wt{P}_0^\perp \Phi'( a ) %
\wt{P}_0^\perp ( \id_\ini \otimes X ) \bigr) = 0,
\]
by \eqref{eqn:diagker}, and then Lemma~\ref{lem:cruc} gives that
\begin{align*}
E^\vac \mmodf{\ppi \comp \Phi}{\tau}( a ) E_{[ X ]} & = %
\sstate\bigl( ( \tau^{1 / 2} \diag + \diag^\perp ) \bigl( %
\nmodf{\Phi}{\tau}( a ) \bigr) ( \id_\ini \otimes X ) \bigr) \\[1ex]
& = E^\vac \ppi\Bigl( ( \tau^{1 / 2} \diag + \diag^\perp ) %
\bigl( \nmodf{\Phi}{\tau}( a ) \bigr) \Bigr) E_{[ X ]};
\end{align*}
similar working produces the identity
\[
E^{[ X ]} \mmodf{\ppi \comp \Phi}{\tau}( a ) E_\vac = %
E^{[ X ]} \ppi\Bigl( ( \tau^{1 / 2} \diag + \diag^\perp ) %
\bigl( \nmodf{\Phi}{\tau}( a ) \bigr) \Bigr) E_\vac.
\]
Finally, if $X$,~$Y \in \ker \state$ then \eqref{eqn:modif} and
Lemma~\ref{lem:cruc} give that
\[
E^{[ X ]} \mmodf{\ppi \comp \Phi}{\tau}( a ) E_{[ Y ]} = %
E^{[ X ]} \ppi\bigl( \wt{P}_0^\perp \Phi'( a ) \wt{P}_0^\perp + %
\tau^{1/ 2} R_1( a, \tau ) \bigr) E_{[ Y ]},
\]
where
\[
R_1( a, \tau ) := %
( \tau^{1/ 2 } \diag + \diag^\perp )%
\bigl( \nmodf{\Phi}{\tau}( a ) \bigr) - %
\wt{P}_0^\perp \Phi'( a ) \wt{P}_0^\perp.
\]
Hence
\begin{align*}
( \mmodf{\ppi \comp \Phi}{\tau} - \psi )( a ) & = %
\Delta^\perp \ppi\bigl( ( \nmodf{\Phi}{\tau} - \Psi )( a ) \bigr) %
\Delta^\perp + \Delta^\perp ( \ppi \comp \diag^\perp )%
\bigl( ( \nmodf{\Phi}{\tau} - \Psi )( a ) \bigr) \Delta \\
& \quad + \Delta ( \ppi \comp \diag^\perp )\bigl( %
( \nmodf{\Phi}{\tau} - \Psi )( a ) \bigr) \Delta^\perp \\
 & \quad + \Delta %
\ppi( \wt{P}_0^\perp ( \nmodf{\Phi}{\tau} - \Psi )( a ) %
\wt{P}_0^\perp ) \Delta + \tau^{1/2} R_2( a, \tau ),
\end{align*}
where
\[
R_2( a, \tau ) := \Delta^\perp \bigl( \ppi \comp \diag \comp %
\nmodf{\Phi}{\tau} \bigr)( a ) \Delta + %
\Delta^\perp \bigl( %
\ppi \comp \diag \comp \nmodf{\Phi}{\tau} \bigr)( a ) \Delta + %
\Delta \ppi\bigl( R_1( a , \tau ) \bigr) \Delta.
\]
The result now follows from Theorem~\ref{thm:qrw}.
\end{proof}

\begin{remark}
Theorem~\ref{thm:main} is an extension of previous results. If
$\state$ is faithful or a vector state then it reduces to
\cite[Theorem~3]{Blt10b} or \cite[Theorem~7.6]{Blt10a}, respectively;
the former theorem has \cite[Theorem~7]{AtJ07a} as a special case,
whereas the latter is a generalisation of Attal and Pautrat's
convergence theorem \cite[Theorem~13]{AtP06}.
\end{remark}

\begin{proposition}\label{prp:indnoise}
Let $X$, $Y \in \ker \state$. If $\psi( a )$ is given by
\eqref{eqn:psidef} then
\begin{align*}
E^\vac \psi( a ) E_\vac & = %
E^\vac \ppi\bigl( \Psi( a ) \bigr) E_\vac = %
\sstate\bigl( \Psi( a ) \bigr), \\[1ex]
E^\vac \psi( a ) E_{[ Y ]} & = %
E^\vac \ppi\bigl( \Psi( a ) \bigr) E_{[ d^\perp( Y ) ]}, \\[1ex]
E^{[ X ]} \psi( a ) E_\vac & = %
E^{[ d^\perp( X ) ]} \ppi\bigl( \Psi( a ) \bigr) E_\vac \\[1ex]
\text{and} \quad %
E^{[ X ]} \psi( a ) E_{[ Y ]} & = %
E^{[ P_0^\perp X ]} \ppi\bigl( \Psi( a ) \bigr) E_{[ P_0^\perp Y ]}.
\end{align*}
Thus if $N := \dim \statesp < \infty$ then there can be no more than
\begin{equation}\label{num:indnoise}
2 ( N k - l ) + ( N - k )^2 k^2
\end{equation}
independent noises in the quantum stochastic differential equation
\eqref{eqn:qsde} satisfied by the limit cocycle $j^\psi$, where
\[
k := \dim \statesp_0 \in \{ 1, \ldots, N \} \qquad \text{and} \qquad %
l := \rank d_0 \in \{ 1, \ldots, k^2 \}.
\]
\end{proposition}
\begin{proof}
If $Y \in \ker \state$ then \eqref{eqn:homstate} implies that
\[
E^\vac \psi( a ) E_{[ Y ]} = %
\sstate\bigl( ( \diag^\perp \comp \Psi )( a ) %
( \id_\ini \otimes Y ) \bigr) = %
\sstate\bigl( \Psi( a ) ( \id_\ini \otimes Y ) \bigr) - %
\sstate\bigl( ( \diag \comp \Psi )( a ) ( \id_\ini \otimes Y ) \bigr)
\]
However, as $\diag$ preserves $\sstate$, it follows from the bimodule
property that
\[
\sstate\bigl( ( \diag \comp \Psi )( a ) %
( \id_\ini \otimes Y ) \bigr) = %
( \sstate \comp \diag )\bigl( ( \diag \comp \Psi )( a ) %
( \id_\ini \otimes Y ) \bigr) = %
( \sstate \comp \diag )\bigl( \Psi( a ) %
\diag( \id_\ini \otimes Y ) \bigr)
\]
and therefore
\[
E^\vac \psi( a ) E_{[ Y ]} = %
\sstate\bigl( \Psi( a ) \diag^\perp( I_\ini \otimes Y ) \bigr) = %
E^\vac \ppi\bigl( \Psi( a ) \bigr) E_{[ d^\perp( Y ) ]},
\]
as required. The other identities are may be established similarly.

Henceforth, suppose that $\statesp$ is finite dimensional. From the
previous working, there can be no more than $2 n_1 + n_2^2$
independent noises in the quantum stochastic differential
equation~\eqref{eqn:qsde}, where
\[
n_1 := \dim \{ [ d^\perp( X ) ] : X \in \ker \state \} %
\qquad \text{and} \qquad %
n_2 := \dim \{ [ P_0^\perp X ] : X \in \ker \state \}.
\]
To find $n_2$, note that
\[
[ P_0^\perp X ] = %
( P_0^\perp \otimes \id_{\conj{\statesp_0}} ) \pi( X ) \vac = %
U_\statesp( P_0^\perp X \State^{1 / 2} ) %
\qquad \text{for all } X \in \bop{\statesp}{},
\]
so that, in particular, $[ P_0^\perp ] = 0$; as $\vac$ is a cyclic
vector for the representation $\pi$, it follows that
\[
n_2 = %
\rank( P_0^\perp \otimes \id_{\conj{\statesp_0}} ) = %
\dim( \statesp_0^\perp \otimes \conj{\statesp_0} ) = ( N - k ) k.
\]
For $n_1$, note first that $d^\perp( \id_\statesp ) = P_0^\perp$ and
$[ P_0^\perp ] = 0$, hence
\[
n_1 = \dim \{ [ d^\perp( X ) ] : X \in \bop{\statesp}{} \} = %
\dim \{ d^\perp( X ) \State^{1 / 2} : X \in \bop{\statesp}{} \}.
\]
Writing
$X = %
\left[\begin{smallmatrix}
 X_0^0 & X_\times^0 \\[0.5ex]
 X_0^\times & X_\times^\times
\end{smallmatrix}\right]$
and
$\State = %
\left[\begin{smallmatrix}
 \State_0^{\vphantom{1 / 2}} \ & 0 \vphantom{X_0^0} \\[1ex]
 0 \ & 0 \vphantom{X_0^0}
\end{smallmatrix}\right]$,
it follows that
\[
d^\perp( X ) \State^{1 / 2} = 
\left[\begin{smallmatrix}
 d_0^\perp( X_0^0 ) & X_\times^0 \\[0.5ex]
 X_0^\times & X_\times^\times
\end{smallmatrix}\right] \, 
\left[\begin{smallmatrix}
 \State_0^{1 / 2} & 0 \vphantom{X_0^0} \\[1ex]
 0 & 0 \vphantom{X_0^0}
\end{smallmatrix}\right] = %
\left[\begin{smallmatrix}
 d_0^\perp( X_0^0 ) \State_0^{1 / 2} & 0 \\[0.5ex]
 X_0^\times \State_0^{1 / 2} & 0
\end{smallmatrix}\right].
\]
As $\state_0$ is faithful, the operator $\State_0^{1 / 2}$ is
invertible, therefore
\[
\dim \{ X_0^\times \State_0^{1 / 2} : %
X_0^\times \in \bop{\statesp_0}{\statesp_0^\perp} \} = %
\dim \bop{\statesp_0}{\statesp_0^\perp} = k ( N - k ).
\]
Similarly,
\[
\{ d_0^\perp( Z ) \State_0^{1 / 2} : Z \in \bop{\statesp_0}{} \} %
\cong \{ d_0^\perp( Z ) : Z \in \bop{\statesp_0}{} \} = %
\{ d_0( Z ) : Z \in \bop{\statesp_0}{} \}^\perp,
\]
where the orthogonal complement is taken with respect to the inner
product
\[
\langle Z, W \rangle := \state_0( Z^* W ) %
\qquad \text{for all } Z, W \in \bop{\statesp_0}{};
\]
the last equality holds because the bimodule property and the fact
that $d_0$ preserves $\state_0$ imply that $d_0$ is a self-adjoint
linear idempotent, \ie an orthogonal projection, on this inner-product
space. As
\[
\dim \{ d_0( Z ) : Z \in \bop{\statesp_0}{} \}^\perp = %
\dim \bop{\statesp_0}{} - \rank d_0 = k^2 - l,
\]
the result now follows.
\end{proof}

\begin{remark}
Suppose $N := \dim \statesp < \infty$ and let $k := \dim \statesp_0$
and $l := \rank d_0$, as in Proposition~\ref{prp:indnoise}. Since
$\mmul = \statesp \otimes \conj{\statesp_0}$, in principle
$\dim \mul = N^2 k^2 - 1$ quantum noises can appear in the quantum
stochastic differential equation \eqref{eqn:qsde}.

If $\state$ is a vector state then $k = 1$ and $l = 1$, so
\eqref{num:indnoise} equals
\[
2 ( N - 1 ) + ( N - 1 )^2 = N^2 - 1,
\]
as expected. At the other extreme, if $\state$ is a faithful state
then \eqref{num:indnoise} equals $2 ( N^2 - l )$.

In general,
\begin{align*}
N^2 k^2 - 1 - \big( 2 ( N k - l ) + ( N - k )^2 k^2 \bigr) & = %
2 N k^3 - k^4 - 2 N k + 2 l - 1 \\[1ex]
 & = ( k^2 - 1 ) \bigl( ( 2 N - k ) k - 1 \bigr) + 2 l - 2
\end{align*}
and this equals zero if and only if $k = 1$. Hence the thermalisation
phenomenon, the loss of noises in the quantum stochastic differential
equation which governs the limit cocycle, occurs exactly when $\state$
is not a vector state.
\end{remark}

\section{Applications}\label{sec:applications}

\begin{notation}
Let $\vna \subseteq \bop{\ini}{}$ be a von~Neumann algebra; recall
that $\vna \matten \bop{\hilb}{} = \vna \uwkten \bop{\hilb}{}$ for any
Hilbert space $\hilb$, and
$\Phi \matten \id_{\bop{\hilb}{}} = \Phi \uwkten \id_{\bop{\hilb}{}}$
for any ultraweakly continuous, $\hilb$-bounded map $\Phi$.

Let $\state$ be a normal state on $\bop{\statesp}{}$, with density
matrix $\State$, and let $\statesp_0 := \overline{\im \State^{1 / 2}}$
as in Section~\ref{sec:concrete}. Suppose $d_0$ is a conditional
expectation on $\bop{\statesp_0}{}$ which preserves the faithful state
$\state_0$ defined in Lemma~\ref{lem:faithfulstate}; let
$\diag_0 := \id_{\bop{\ini}{}} \uwkten d_0$ and
$\sstate_0 := \id_{\bop{\ini}{}} \uwkten \state_0$.
\end{notation}

\subsection{Hudson--Parthasarathy evolutions}

The following theorem is a generalisation of both
\cite[Remark~7]{Blt10b} and the well-known Hudson--Parthasarathy
conditions for processes to be isometric, co-isometry or unitary.

\begin{theorem}\label{thm:hp}
Let $F \in \vna \uwkten \bop{\statesp}{}$ and define
\[
\Psi : \vna \to \vna \uwkten \bop{\statesp}{}; \ %
a \mapsto ( a \otimes \id_\statesp ) F.
\]
If $\psi : \vna \to \vna \uwkten \bop{\mmul}{}$ is given
by~\eqref{eqn:psidef} then $\psi( a ) = ( a \otimes \id_\mmul ) G$ for
all~$a \in \vna$, where
\begin{equation}\label{eqn:Gdef}
G := \Delta^\perp \ppi( F ) \Delta^\perp + %
\Delta^\perp ( \ppi \comp \diag^\perp )( F ) \Delta + %
\Delta ( \ppi \comp \diag^\perp )( F ) \Delta^\perp + %
\Delta \ppi( \wt{P}_0^\perp F \wt{P}_0^\perp ) \Delta \in %
\vna \uwkten \bop{\mmul}{}.
\end{equation}
The cocycle $j^\psi$ is such that
$j^\psi_t( a ) = ( a \otimes \id_\mmul ) X_t$ for all $a \in \vna$ and
$t \in \R_+$, where the adapted $\ini$~process $X = ( X_t )_{t \ge 0}$
satisfies the right Hudson--Parthasarathy equation
\begin{equation}\label{eqn:rightqsde}
X_0 = \id_{\ini \otimes \fock}, \quad %
\rd X_t = \rd \Lambda_G( t ) \, X_t \qquad \text{for all } t \in \R_+.
\end{equation}
The process $X$ is composed of isometric, co-isometric or unitary
operators if and only if
\begin{equation}\label{eqn:fmat}
F = \begin{bmatrix}
 -\I ( H_d + H_o ) - \hlf K & -D^* V \\[1ex]
 D & V - I_{\ini \otimes \statesp_0^\perp} \end{bmatrix},
\end{equation}
where
\begin{mylist}
\item[\tu{(i)}] $H_d$, $H_o \in \vna \uwkten \bop{\statesp_0}{}$ are
self adjoint, with $H_d = \diag_0( H_d )$ and
$H_o = \diag_0^\perp( H_o )$,
\item[\tu{(ii)}] $K \in \vna \uwkten \bop{\statesp_0}{}$ is self
adjoint, with $K = \diag_0( K )$ and
$\sstate_0( K ) = \sstate_0( H_o^2 + D^* D )$,
\item[\tu{(iii)}]
$D \in \vna \uwkten \bop{\statesp_0}{\statesp_0^\perp}$
\item[and \tu{(iv)}] $V \in \vna \uwkten \bop{\statesp_0^\perp}{}$ is
isometric, co-isometric or unitary, respectively.
\end{mylist}
\end{theorem}
\begin{proof}
The first claim is immediate, and the second follows from
\cite[Proof of Theorem~7.1]{LiW00}. For the final part, recall that
$\sstate \comp \diag = \sstate$, by Proposition~\ref{prp:condexp}; it
follows from this, \eqref{eqn:homstate} and \eqref{eqn:sstateker} that
\[
G + G^* + G^* \Delta G = \Delta^\perp \ppi( F_1 ) \Delta^\perp + %
\Delta^\perp \ppi( F_2 ) \Delta + %
\Delta \ppi( F_2^* ) \Delta^\perp + \Delta \ppi( F_3 ) \Delta,
\]
where
\begin{align*}
F_1 & := \diag\bigl( F + F^* + %
\diag^\perp( F^* ) \diag^\perp( F ) \bigr), \quad
F_2 := \wt{P}_0 ( \diag^\perp( F + F^* ) + \diag^\perp( F^* ) %
\wt{P}_0^\perp F \wt{P}_0^\perp ) \\[1ex]
\text{and} \quad F_3 & := %
\wt{P}_0^\perp ( F + F^* + F^* \wt{P}_0^\perp F ) \wt{P}_0^\perp,
\end{align*}
so
\[
G + G^* + G^* \Delta G = 0 \iff %
\Delta^\perp \ppi( F_1 ) \Delta^\perp = %
\Delta^\perp \ppi( F_2 ) \Delta = %
\Delta \ppi( F_3 ) \Delta = 0.
\]
Let
\smash[b]{$F = \left[ \begin{smallmatrix} A + B \ & C \\[0.5ex]
D & E \end{smallmatrix} \right]$},
where $\diag_0( A ) = A$ and $\diag_0^\perp( B ) = B$; after some
working, it may be shown that
\begin{align*}
F_1 & = \begin{bmatrix}
 \sstate_0( A + A^* + B^* B + D^* D ) & 0 \\[1ex]
 0 & 0 \end{bmatrix}, \quad %
F_2 = \begin{bmatrix} B + B^* & C + D^* + D^* E \\[1ex]
0 & 0 \end{bmatrix} \\[1ex]
\text{and} \quad F_3 & = %
\begin{bmatrix} 0 & 0 \\[1ex] 0 & E + E^* + E^* E \end{bmatrix}.
\end{align*}
If $X \in \bop{\statesp}{}$ then
\[
E^\vac \ppi( F_3 ) E_{[ X ]} = %
\sstate\bigl( \wt{P}_0 F_3 ( \id_\ini \otimes X ) \bigr) = 0 = %
\sstate\bigl( ( \id_\ini \otimes X )^* F_3 \wt{P}_0 \bigr) = %
E^{[ X ]} \ppi( F_3 ) E_\vac,
\]
so $\Delta^\perp \ppi( F_3 ) = \ppi( F_3 ) \Delta^\perp = 0$ and
therefore $\Delta \ppi( F_3 ) \Delta = \ppi( F_3 )$. Hence
\[
\Delta \ppi( F_3 ) \Delta = 0 \iff E + E + E^* E = 0 \iff %
V^* V = \id_{\ini \otimes \statesp_0^\perp},
\]
where $V = E + I_{\ini \otimes \statesp_0^\perp}$. Next, note that
\[
E^\vac \ppi( F_2 ) E_\vac = %
( \sstate \comp \diag )\bigl( \diag^\perp( F + F^* ) \bigr) + %
\sstate( \diag^\perp( F^* ) \diag( F ) %
\wt{P}_0^\perp F \wt{P}_0^\perp \wt{P}_0 ) = 0,
\]
so $\Delta^\perp \ppi( F_2 ) \Delta = \Delta^\perp \ppi( F_ 2 )$. If
$Y = \left[\begin{smallmatrix} Y_0^0 & Y_\times^0 \\[0.5ex]
 Y_0^\times & Y_\times^\times \end{smallmatrix}\right] \in %
\bop{\statesp}{}$
then
\[
E^\vac \ppi( F_2 ) E_{[ Y ]} = \begin{bmatrix}
 \sstate_0\bigl( ( B + B^* ) ( \id_\ini \otimes Y_0^0 ) + %
( C + D^* V ) ( \id_\ini \otimes Y_0^\times ) \bigr) & 0 \\[1ex]
 0 & 0
\end{bmatrix},
\]
therefore $\Delta^\perp \ppi( F_2 ) \Delta = 0$ if and only if
\[
 \sstate_0\bigl( ( B + B^* ) ( \id_\ini \otimes Y_0^0 ) \bigr) = %
\sstate_0\bigl( ( C + D^* V ) %
( \id_\ini \otimes Y_0^\times ) \bigr) = 0
\]
for all $Y_0^0 \in \bop{\statesp_0}{}$ and
$Y_0^\times \in \bop{\statesp_0}{\statesp_0^\perp}$. Suppose
$T \in \bop{\ini \otimes \statesp_0}{}$ is such that
\[
 \sstate_0\bigl( T ( \id_\ini \otimes Y_0^0 ) \bigr) = 0 \qquad %
\text{for all } Y_0^0 \in \bop{\statesp_0}{};
\]
with the notation as in Lemma~\ref{lem:faithfulstate},
\[
0 = \langle u, \sstate_0\bigl( T %
( \id_\ini \otimes Y_0^0 ) \bigr) v \rangle = %
\sum_{j \in J} \lambda_j \langle u \otimes e_j, T %
( v \otimes Y_0^0 e_j ) \rangle \qquad \text{for all } u, v \in \ini,
\]
where $\lambda_j > 0$ for all $j \in J$ and $\{ e_j : j \in J \}$ is
an orthonormal basis for $\statesp_0$. With~$Y_0^0 = \dyad{y}{e_j}$
for arbitrary $y \in \statesp_0$ and $j \in J$, this gives that
\[
T\bigl( \ini \algten \statesp_0 \bigr) \perp %
\ini \algten \lin\{ e_j : j \in J \}
\]
and therefore $T = 0$. Taking $T = B + B^*$ and
$T = ( C + D^* V ) \dyad{x}{y}$, where $x \in \statesp_0^\perp$ 
and~$y \in \statesp_0$ are arbitrary, it follows that
\[
\Delta^\perp \ppi( F_2 ) \Delta = 0 \iff B + B^* = 0 %
\quad \text{and} \quad C + D^* V = 0.
\]
Finally,
\[
\Delta^\perp \ppi( F_1 ) \Delta^\perp = 0 \iff %
\sstate_0( A + A^* ) = -\sstate_0( B^* B + D^* D )
\]
and the result now follows from \cite[Theorem~7.5]{LiW00}.
\end{proof}

\begin{remark}
By definition, the adapted $\ini$~process $X$ satisfies
equation~\eqref{eqn:rightqsde} if and only if
\[
\langle u\evec{f}, %
( X_t - \id_{\ini \otimes \fock} ) v \evec{g} \rangle = %
\int_0^t \langle u \evec{f}, %
( E^{\widehat{f( s )}} G E_{\widehat{g( s )}} \otimes \id_\fock ) %
X_s v \evec{g} \rangle \std s
\]
for all $u$,~$v \in \ini$, $f$,~$g \in \elltwo$ and $t \in \R_+$.
\end{remark}

\begin{notation}
Define the \emph{decapitated exponential functions}
\[
\exp_1( z ) = \sum_{n = 1}^\infty \frac{z^{n - 1}}{n!} %
\quad \text{and} \quad %
\exp_2( z ) = \sum_{n = 2}^\infty \frac{z^{n - 2}}{n!} \qquad %
\text{for all } z \in \C.
\]
Note that
\begin{equation}\label{eqn:expids}
\exp_1( z ) \exp( -z ) = \exp_1( -z ) \quad \text{and} \quad %
\exp_1( z ) \exp_1( -z ) = \exp_2( z ) + \exp_2( -z )
\end{equation}
for all $z \in \C$.
\end{notation}

\begin{theorem}\label{thm:exhp}
Let the \emph{total Hamiltonian}
\[
H_t( \tau ) = \begin{bmatrix}
 H_d + \tau^{-1 / 2} H_o & \tau^{-1/ 2} L^* \\[1ex]
 \tau^{-1 / 2} L & \tau^{-1} H_\times \end{bmatrix} + %
\begin{bmatrix}
 R_0^0( \tau ) & \tau^{-1 / 2 } R_\times^0( \tau ) \\[1ex]
 \tau^{-1 / 2} R_0^\times( \tau ) & %
\tau^{-1} R_\times^\times( \tau ) \end{bmatrix} \in %
\vna \uwkten \bop{\statesp}{}
\]
for all $\tau > 0$, where
\begin{mylist}
\item[\tu{(i)}] the self-adjoint operators $H_d$,
$H_o \in \vna \uwkten \bop{\statesp_0}{}$ are such that
$\diag_0( H_d ) = H_d$ and $\diag_0^\perp( H_o )= H_o$,
\item[\tu{(ii)}]
$L \in \vna \uwkten \bop{\statesp_0}{\statesp_0^\perp}$,
\item[\tu{(iii)}] $H_\times \in \vna \uwkten \bop{\statesp_0^\perp}{}$
is self adjoint
\item[and \tu{(iv)}] the functions $R_0^0$, $R_0^\times$, $R_\times^0$
and $R_\times^\times$ are such that
\end{mylist}
\begin{align}
\text{\tu{(a)} } & R_0^0( \tau ) = R_0^0( \tau )^*, \ %
R_\times^0( \tau ) = R_0^\times( \tau )^* \ \text{ and } \ %
R_\times^\times( \tau ) = R_\times^\times( \tau )^* \ %
\text{for all } \tau > 0, \nonumber \\[1ex] \label{eqn:errorbound}
\text{\tu{(b)} } & \text{the function } %
\tau \mapsto \| R_0^0( \tau ) \| %
\text{ is bounded on a neighbourhood of } 0 \\[1ex]
\label{eqn:errorslimits} \text{and \tu{(c)} } & %
\lim_{\tau \to 0 } \diag_0\bigl( R_0^0( \tau ) \bigr) = %
\lim_{\tau \to 0 } R_0^\times( \tau ) = %
\lim_{\tau \to 0 } R_\times^\times( \tau ) = 0, \\
& \text{where the convergence holds in the norm topology}. \nonumber
\end{align}
Then the completely isometric map
\[
\Phi( \tau ) : \vna \to \vna \uwkten \bop{\statesp}{}; \ %
a \mapsto %
( a \otimes \id_\statesp ) \exp\bigl( -\I \tau H_t( \tau ) \bigr)
\]
is such that
$\| \nnmodf{\diag}{\Phi( \tau )}{\tau} - \Psi \|_\cb \to 0$
as $\tau \to 0$, where
\[
\Psi : \vna \to \vna \uwkten \bop{\statesp}{}; \ %
a \mapsto ( a \otimes \id_\statesp ) F
\]
and
\begin{equation}\label{eqn:hpFgen}
F = \begin{bmatrix}
 -\I ( H_d + H_o ) - %
\diag_0( \hlf H_o^2 + L^* \exp_2( -\I H_\times ) L ) & %
 -\I L^* \exp_1( -\I H_\times ) \\[1ex]
-\I \exp_1( -\I H_\times ) L & %
\exp( -\I H_\times ) - I_{\ini \otimes \statesp_0^\perp}
\end{bmatrix}.
\end{equation}
Consequently, $J^{\ppi \comp \Phi( \tau ), \tau} \ttocb j^\psi$, where
the completely bounded map
$\psi : \vna \to \vna \uwkten \bop{\mmul}{}$ is as defined
in~\eqref{eqn:psidef}. The adapted $\ini$~process
$( U_t := j^\psi_t( \id_\ini ) )_{t \in \R_+}$ is unitary for all
$t \in \R_+$ and such that
$j^\psi_t( a ) = ( a \otimes \id_\mmul ) U_t$ for all $t \in \R_+$
and $a \in \vna$.
\end{theorem}
\begin{proof}
Let $G := \tau H_t( \tau ) = A + \tau^{1 / 2} B + \tau C$, where
\[
A = \begin{bmatrix} 0 & 0 \\[1ex]
 0 & H_\times + R_\times^\times( \tau ) \end{bmatrix}, \quad %
B = \begin{bmatrix} H_o & L^* + R_\times^0( \tau ) \\[1ex]
 L + R_0^\times( \tau ) & 0 \end{bmatrix} \quad \text{and} \quad %
C = \begin{bmatrix} H_d + R_0^0( \tau ) & 0 \\[1ex]
 0 & 0 \end{bmatrix};
\]
by \eqref{eqn:errorbound} and \eqref{eqn:errorslimits}, there exists
$\tau_0 \in ( 0, 1 )$ such that
\[
c := %
\sup\{ \| A \|, \| B \|, \| C \| : 0 < \tau < \tau_0 \} < \infty.
\]
Then
\[
\nmodf{\Phi( \tau )}{\tau}( a ) = ( a \otimes \id_\statesp ) %
\smash[b]{\sum_{n = 1}^\infty \frac{( -\I )^n}{n!} m( G^n )} %
\qquad \text{for all } \tau > 0,
\]
where the linear map
\[
m : T \mapsto \wt{P}_0 %
( \tau^{-1} \delta + \tau^{-1 / 2} \delta^\perp )( T ) \wt{P}_0 + %
\tau^{-1 / 2} \wt{P}_0 T \wt{P}_0^\perp + %
\tau^{-1 / 2} \wt{P}_0^\perp T \wt{P}_0^\perp + %
\wt{P}_0^\perp T \wt{P}_0^\perp.
\]
Note that
\begin{multline*}
G^n = A^n + %
\smash{\tau^{1 / 2} \sum_{j = 0}^{n - 1} A^j B A^{n - 1 - j} + %
\tau \sum_{j = 0}^{n - 1} A^j C A^{n - 1 - j}} \\
 + \tau \sum_{j = 0}^{n - 2} %
\sum_{k = 0}^{n - 2 - j} A^j B A^k B A^{n - 2 - j - k} + %
\tau^{3 / 2} D_n
\end{multline*}
for all $n \ge 1$, where $\| D_n \| \le 3^n c^n$ for
all~$\tau \in ( 0, \tau_0 )$. As $A C = C A = 0$ and $A B A = 0$, this
simplifies to give that
\[
G^n = A^n + \tau^{1 / 2} ( B A^{n - 1} + A^{n - 1} B ) + %
\tfn{n \ge 3} \tau B A^{n - 2} B + %
\tau \sum_{j = 0}^{n - 2} A^j B^2 A^{n - 2 - j} + \tau^{3 / 2} D_n
\]
for all $n \ge 2$. (Here and below, the expression $\tfn{P}$ has the
value $1$ if $P$ is true and $0$ if $P$ is false.) Furthermore, if
$p \ge 1$, $0 \le j \le p$ and
\[
r_\tau( T ) := \wt{P}_0 \diag^\perp( T ) \wt{P}_0 + %
\wt{P}_0 T \wt{P}_0^\perp + \wt{P}_0^\perp T \wt{P}_0 + %
\tau^{1 / 2} \wt{P}_0^\perp T \wt{P}_0^\perp
\]
then
\begin{alignat*}{2}
& m( A^p )= A^p, & &
m\big( \tau^{1 / 2} ( B A^p + A^p B ) \bigr) = %
B A^p + A^p B, \\[1ex]
& m( \tau B^2 ) = \diag( B^2 ) + \tau^{1 / 2} r_\tau( B^2 ), & &
m( \tau B A^p B ) = \diag( B A^p B ) + %
\tau^{1 / 2} \diag^{\perp}( B A^p B ) \\[1ex]
\text{and} \quad & m( \tau A^j B^2 A^{p - j} ) = %
\tau^{1 / 2} r_\tau( A^j B^2 A^{p - j} ).
\end{alignat*}
Hence, omitting the argument $\tau$ from $R_0^0$, $R_0^\times$,
$R_\times^0$ and $R_\times^\times$ for brevity,
\[
m( G ) = \begin{bmatrix}
 H_d + H_o & L^* + R_\times^0 \\[1ex]
 L + R_0^\times & H_\times + R_\times^\times
\end{bmatrix} + D_1', \qquad \text{where} \quad D_1' = %
\begin{bmatrix}
 \diag_0( R_0^0 ) + \tau^{1 / 2} \diag_0^\perp( R_0^0 ) & 0 \\[1ex]
 0 & 0 \end{bmatrix},
\]
and
\begin{align*}
m( G^n ) & = A^n + B A^{n - 1} + A^{n - 1} B + %
\diag( B A^{n - 2} B ) + \tau^{1 / 2} D_n' \\[1ex]
 & = \begin{bmatrix} \tfn{n = 2} H_o^2 + %
( L^* + R_\times^0 )( H_\times + R_\times^\times )^{n - 2}%
( L + R_0^\times ) & ( L^* + R_\times^0 )%
( H_\times + R_\times^\times )^{n - 1} \\[1ex]
( H_\times + R_\times^\times )^{n - 2}( L + R_0^\times ) & %
( H_\times + R_\times^\times )^n
\end{bmatrix} \\
 & \qquad +\tau^{1 / 2} D_n'
\end{align*}
for all $n \ge 2$, where
\[
D_n' = \tfn{n \ge 3} \diag^\perp( B A^{n - 2} B ) + %
\sum_{j = 0}^{n - 2} r_\tau( A^j B^2 A^{n - 2 - j} ) + m( \tau D_n );
\]
in particular, if $n \ge 2$ and $\tau \in ( 0, \tau_0 )$ then
\[
\| D_n' \| \le 2 c^n + 5 ( n - 1 ) c^n + 6 ( 3 c )^n = %
 ( 5 n - 3 + 2 \cdot 3^{n + 1} ) c^n.
\]
An $M$-test argument now gives that
$\| \nmodf{\Phi( \tau )}{\tau} - \Psi\|_\cb \to 0$ as $\tau \to 0$ and
therefore $J^{\ppi \comp \Phi( \tau ), \tau} \ttocb j^\psi$, by
Theorem~\ref{thm:main}. Using the identities \eqref{eqn:expids}, it is
readily verified that~$F$ satisfies the unitarity conditions of
Theorem~\ref{thm:hp}; in the notation of that theorem, but with~$H_d$
and~$H_o$ there replaced by~$H_d'$ and~$H_o'$,
\begin{alignat*}{2}
H_d' & = H_d - \mbox{$\frac{\I}{2}$} \diag_0( L^* \bigl( %
\exp_2( -\I H_\times ) - \exp_2( \I H_\times ) \bigr) L ), %
& \qquad H_o' & = H_o, \\[1ex]
K & = \diag_0( H_o^2 + L^* \bigl( \exp_2( -\I H_\times ) + %
\exp_2( \I H_\times ) \bigr) L ), %
& D & = -\I \exp_1( -\I H_\times ) L \\[1ex]
\text{and} \quad V & = \exp( -\I H_\times ). \tag*{\qedhere}
\end{alignat*}
\end{proof}

\begin{remark}
When the state $\state$ is faithful or a vector state,
Theorem~\ref{thm:exhp} is a generalisation of \cite[Theorem~4]{Blt10b}
or \cite[Theorem~19]{AtP06}, respectively; for the latter case, see
also \cite[Theorem~4.1]{Gou04}.
\end{remark}

The following example is the simplest which illustrates the various
features of Theorem~\ref{thm:exhp}.

\begin{example}
Suppose $\statesp = \C^3$ and take the density matrix
\[
\State = \begin{bmatrix}
 \lambda_1 & 0 & 0 \\[1ex]
 0 & \lambda_2 & 0 \\[1ex]
 0 & 0 & 0
\end{bmatrix} \in M_3( \C ), \qquad \text{where }
\lambda_1, \lambda_2 \in ( 0, 1 ) \text{ are such that }
\lambda_1 + \lambda_2 = 1.
\]
Then $\statesp_0 = \C^2$; let the
$\state_0$-preserving conditional expectation
\[
d_0 : M_2( \C ) \to M_2( \C ); \ %
\begin{bmatrix} x & y \\[0.5ex]
 z & w \end{bmatrix} \mapsto %
\begin{bmatrix}
 x & 0 \\[0.5ex]
 0 & w
\end{bmatrix}.
\]
Let $e_{i j} \in M_3( \C )$ be the elementary matrix with $1$ in the
$( i, j )$~entry and $0$ elsewhere, let
\[
f_{i j} = \lambda_j^{-1 / 2} e_{i j} \qquad \text{for } i = 1, 2, 3 %
\text{ and } j = 1, 2,
\]
and let $\{ e_i : i = 1, 2, 3 \}$ be the canonical basis of $\C^3$, so
that~$\mmul$ has the basis
\[
\{ [ f_{i j} ] = e_i \otimes \conj{e_j} : i = 1, 2, 3, \ j = 1, 2 \}.
\]
Note also that $d( f_{i j} ) = 0$ unless $i = j = 1$ or $i = j = 2$,
and $\{ [ d^\perp( X ) ] : X \in \ker \state \}$ has basis
\[
\bigl\{ [ d^\perp( f_{i j} ) ] = e_i \otimes \conj{e_j} : %
( i, j ) \in \{ ( 1, 2 ), ( 2, 1 ), ( 3, 1 ), ( 3, 2 ) \} \bigr\};
\]
similarly, $P_0 f_{3 k} = 0$ for $k = 1$ and $k = 2$, and
$\{ [ P_0^\perp X ] : X \in \ker \state \}$ has the basis
\[
\{ [ P_0^\perp f_{3 k} ] = e_3 \otimes \conj{e_k} : k = 1, 2 \}.
\]
If $H_d$, $H_o$, $L$ and $H_\times$ are as in Theorem~\ref{thm:exhp}
then
\[
H_d = \begin{bmatrix} b & 0 \\[0.5ex]
 0 & c \end{bmatrix}, \, %
H_o = \begin{bmatrix} 0 & g^* \\[0.5ex]
 g & 0 \end{bmatrix} \in M_2( \vna ), \quad
L = \begin{bmatrix} l & m \end{bmatrix} \in M_{1, 2}( \vna ) %
\quad \text{and} \quad H_\times = h \in \vna,
\]
where $b$, $c$, $h \in \vna$ are self adjoint. With the notation of
Theorem~\ref{thm:exhp},
\[
H_t( \tau ) = \begin{bmatrix}
 b & \tau^{-1 / 2} g^* & \tau^{-1 / 2} l^* \\[1ex]
 \tau^{-1 / 2} g & c & \tau^{-1 / 2} m^* \\[1ex]
 \tau^{-1 / 2} l & \tau^{-1 / 2} m & \tau^{-1} h
\end{bmatrix} \qquad \text{for all } \tau > 0
\]
and
\[
F = %
\begin{bmatrix}
 -\I b - \hlf g^* g - l^* \exp_2( -\I h ) l & %
-\I g^* & -\I l^* \exp_1( -\I h ) \\[1ex]
-\I g & -\I c - \hlf g^* g - m^* \exp_2( -\I h ) m & %
-\I m^* \exp_1( -\I h ) \\[1ex]
 -\I \exp_1( -\I h ) l & -\I \exp_1( -\I h ) m & %
\exp( -\I h ) - \id_\ini 
\end{bmatrix}.
\]
As
$\vac = %
\lambda_1^{1 / 2} e_1 \otimes \conj{e_1} + %
\lambda_2^{1 / 2} e_2 \otimes \conj{e_2}$,
it follows that
\begin{align*}
E^\vac \psi( a ) E_\vac & = a %
( \lambda_1 F_{1 1} + \lambda_2 F_{2 2} ), \\[1ex]
E^{[ f_{i j} ]} \psi( a ) E_\vac & = %
\lambda_j^{1 / 2} a F_{i j} \\[1ex]
E^\vac \psi( a ) E_{[ f_{i j} ]} & = \lambda_j^{1 / 2} a F_{j i} \\[1ex]
\text{and} \quad E^{[ f_{3 k} ]} \psi( a ) E_{[ f_{3 l} ]} & = %
\tfn{k = l} \, a F_{3 3}
\end{align*}
for all $( i, j ) \in \{ ( 1, 2 ), ( 2, 1 ), ( 3, 1 ) , ( 3, 2 ) \}$
and $k$, $l \in \{ 1, 2 \}$, where~$F_{p q}$ denotes the
$( p, q )$~entry of the matrix~$F$ and $\tfn{k = l}$ equals $1$ if
$k = l$ and $0$ otherwise.

In particular, there are 10 independent quantum noises in the quantum
stochastic differential equations satisfied by the limit
cocycle~$j^\psi$ and the unitary process $U$ given by
Theorem~\ref{thm:exhp}, so the upper bound \eqref{num:indnoise} is not
achieved: in this case, the upper bound equals
\[
2 ( 3 \times 2 - 2 ) + ( 3 - 2 )^2 2^2 = 12.
\]
\end{example}

\subsection{Evans--Hudson evolutions}

The following result is a generalisation of \cite[Remark~8]{Blt10b}.

\begin{theorem}\label{thm:eh}
For any $F \in \vna \uwkten \bop{\statesp}{}$, let
\begin{multline}\label{eqn:ehPsidef}
\Psi : \vna \to \vna \uwkten \bop{\statesp}{}; \ %
a \mapsto ( a \otimes \id_\statesp ) F + %
F^* ( a \otimes \id_\statesp ) + %
\diag\bigl( \diag^\perp( F )^* %
( a \otimes \id_\statesp ) \diag^\perp( F ) \bigr) \\
+ F^* \wt{P}_0^\perp ( a \otimes \id_\statesp ) \wt{P}_0^\perp F - %
\wt{P}_0 F^* \wt{P}_0^\perp ( a \otimes \id_\statesp ) %
\wt{P}_0^\perp F \wt{P}_0
\end{multline}
and let $G \in \vna \uwkten \bop{\mmul}{}$ be given
by~\eqref{eqn:Gdef}. Then $\psi : \vna \to \vna \uwkten \bop{\mmul}{}$
as defined in~\eqref{eqn:psidef} is such that
\begin{equation}\label{eqn:ehpsi}
\psi( a ) = ( a \otimes \id_\mmul ) G + %
G^* ( a \otimes \id_\mmul ) + %
G^* \Delta ( a \otimes \id_\mmul ) \Delta G \qquad %
\text{for all } a \in \vna.
\end{equation}
The cocycle $j^\psi$ is such that
$j^\psi_t( a ) = X_t^* ( a \otimes \id_\mmul ) X_t$ for all
$a \in \vna$ and
$t \in \R_+$, where the adapted $\ini$~process $X = ( X_t )_{t \ge 0}$
satisfies the right Hudson--Parthasarathy
equation~\eqref{eqn:rightqsde}.
\end{theorem}
\begin{proof}
Using Theorem~\ref{thm:hp}, linearity and the adjoint, it suffices to
show that if
\[
\Upsilon( a ) = \diag\bigl( \diag^\perp( F )^* %
( a \otimes \id_\statesp ) \diag^\perp( F ) \bigr) + %
F^* \wt{P}_0^\perp ( a \otimes \id_\statesp ) \wt{P}_0^\perp F - %
\wt{P}_0 F^* \wt{P}_0^\perp ( a \otimes \id_\statesp ) %
\wt{P}_0^\perp F \wt{P}_0
\]
then
\begin{align*}
G^* \Delta ( a \otimes \id_\mmul ) \Delta G & = %
\Delta^\perp ( \ppi \comp \Upsilon )( a ) \Delta^\perp + %
\Delta^\perp ( \ppi \comp \diag^\perp \comp \Upsilon )( a ) %
\Delta \\
& \hspace{6em} + %
\Delta ( \ppi \comp \diag^\perp \comp \Upsilon )( a ) %
\Delta^\perp + \Delta \ppi( \wt{P}_0^\perp \Upsilon( a ) %
\wt{P}_0^\perp ) \Delta \\[1ex]
 & = \Delta^\perp ( \ppi \comp \Upsilon )( a ) \Delta^\perp + %
\Delta^\perp \ppi\bigl( \wt{P}_0 %
( \diag^\perp \comp \Upsilon )( a ) \bigr) %
\Delta \\
& \hspace{6em} + \Delta %
\ppi\bigl( ( \diag^\perp \comp \Upsilon )( a ) \wt{P}_0 \bigr) %
\Delta^\perp + %
\Delta \ppi ( \wt{P}_0^\perp \Upsilon( a ) \wt{P}_0^\perp ) \Delta,
\end{align*}
where the latter equality follows by using \eqref{eqn:homstate}
together with the identities $\sstate \comp \diag^\perp = 0$ and
\[
\sstate( \wt{P}_0 T ) = \sstate( T \wt{P}_0 ) = \sstate( T ) %
\quad \text{for all } T \in \bop{\ini \otimes \statesp}{}.
\]
Letting
$F = \left[ \begin{smallmatrix} X & Y \\[0.5ex]
Z & W \end{smallmatrix}\right]$,
a little algebra shows that
\begin{align}
\diag\bigl( \diag^\perp( F )^* ( a \otimes \id_\statesp ) %
\diag^\perp( F ) \bigr) & = %
\begin{bmatrix}
\diag_0\bigl( \diag_0^\perp( X )^* ( a \otimes \id_{\statesp_0} ) %
\diag_0^\perp( X ) + %
Z^* ( a \otimes \id_{\statesp_0^\perp} ) Z \bigr) & 0 \\[1ex]
 0 & 0
\end{bmatrix}, \nonumber \\[1ex]
F^* \wt{P}_0^\perp ( a \otimes \id_\statesp ) \wt{P}_0^\perp F & = %
\begin{bmatrix} Z^* ( a \otimes \id_{\statesp_0^\perp} ) Z & %
Z^* ( a \otimes \id_{\statesp_0^\perp} ) W \\[1ex]
W^* ( a \otimes \id_{\statesp_0^\perp} ) Z & %
W^* ( a \otimes \id_{\statesp_0^\perp} ) W \end{bmatrix}
\nonumber \\[1ex]\label{eqn:Upsilon}
\text{and} \quad \Upsilon( a ) & = %
\begin{bmatrix}
\diag_0\bigl( \diag_0^\perp( X )^* ( a \otimes \id_{\statesp_0} ) %
\diag_0^\perp( X ) + %
Z^* ( a \otimes \id_{\statesp_0^\perp} ) Z \bigr) & %
Z^* ( a \otimes \id_{\statesp_0^\perp} ) W \\[1ex]
W^* ( a \otimes \id_{\statesp_0^\perp} ) Z & %
W^* ( a \otimes \id_{\statesp_0^\perp} ) W \end{bmatrix}.
\end{align}
Furthermore, with $G$ given by \eqref{eqn:Gdef}, a short calculation
shows that
\begin{multline*}
G^* \Delta ( a \otimes \id_\mmul ) \Delta G \\[1ex]
= \Delta^\perp \ppi\bigl( %
\diag^\perp( F )^* ( a \otimes \id_\statesp ) \diag^\perp( F ) %
\bigr) \Delta^\perp + \Delta^\perp \ppi\bigl( \wt{P}_0 %
\diag^\perp( F )^* ( a \otimes \id_\statesp ) %
\wt{P}_0^\perp F \wt{P}_0^\perp \bigr) \Delta \\[1ex]
+ \Delta \ppi\bigl( %
\wt{P}_0^\perp F^* \wt{P}_0^\perp ( a \otimes \id_\statesp ) %
\diag^\perp( F ) \wt{P}_0 \bigr) \Delta^\perp + %
\Delta \ppi( \wt{P}_0^\perp F^* \wt{P}_0^\perp %
( a \otimes \id_\statesp ) \wt{P}_0^\perp F \wt{P}_0^\perp ) \Delta.
\end{multline*}
Now,
\[
E^\vac \ppi\bigl( \Upsilon( a ) \bigr) E_\vac = %
( \sstate \comp \delta )\bigl( \diag^\perp( F )^* %
( a \otimes \id_\statesp ) \diag^\perp( F ) \bigr) = %
E^\vac \ppi\bigl( \diag^\perp( F )^* ( a \otimes \id_\statesp ) %
\diag^\perp( F ) \bigr) E_\vac
\]
and, since
\[
\diag^\perp( F )^* ( a \otimes \id_\statesp ) \wt{P}_0^\perp F = %
F^* \wt{P}_0^\perp ( a \otimes \id_\statesp ) \wt{P}_0^\perp F = %
F^* \wt{P}_0^\perp ( a \otimes \id_\statesp ) \diag^\perp( F ),
\]
so
\[
\wt{P}_0 \delta^\perp\bigl( \Upsilon( a ) \bigr) = %
\begin{bmatrix} 0 & Z^* ( a \otimes \id_{\statesp_0^\perp} ) W \\[1ex]
0 & 0 \end{bmatrix} = %
\wt{P}_0 \diag^\perp( F )^* ( a \otimes \id_\statesp ) %
\wt{P}_0^\perp F \wt{P}_0^\perp
\]
and
\[
\delta^\perp\bigl( \Upsilon( a ) \bigr) \wt{P}_0 = %
\begin{bmatrix} 0 & 0 \\[1ex]
W^* ( a \otimes \id_{\statesp_0^\perp} ) Z & 0 \end{bmatrix} = %
\wt{P}_0^\perp F^* \wt{P}_0^\perp ( a \otimes \id_\statesp ) %
\diag^\perp( F ) \wt{P}_0.
\]
Finally, as
$\wt{P}_0^\perp \Upsilon( a ) \wt{P}_0^\perp = %
\wt{P}_0^\perp F^* \wt{P}_0^\perp ( a \otimes \id_\statesp ) %
\wt{P}_0^\perp F \wt{P}_0^\perp$,
the first result holds as claimed. The second is an immediate
consequence of \cite[Theorem~7.4]{LiW00}.
\end{proof}

\begin{theorem}\label{thm:ehex}
Let $H_t( \tau )$ be defined as in Theorem~\ref{thm:exhp} for
all~$\tau > 0$. Then the normal $*$-homomorphism
\[
\Phi( \tau ) : \vna \to \vna \uwkten \bop{\statesp}{}; \ %
a \mapsto \exp\bigl( \I \tau H_t( \tau ) \bigr) %
( a \otimes \id_\statesp ) \exp\bigl( -\I \tau H_t( \tau ) \bigr)
\]
is such that $\| \nmodf{\Phi( \tau )}{\tau} - \Psi \|_\cb \to 0$ as
$\tau \to 0$, where $\Psi : \vna \to \vna \uwkten \bop{\statesp}{}$ is
as defined in~\eqref{eqn:ehPsidef} and $F$ is given
by~\eqref{eqn:hpFgen}.

Hence $J^{\ppi \comp \Phi( \tau ), \tau} \ttocb j^\psi$, where the
completely bounded map $\psi : \vna \to \vna \uwkten \bop{\mmul}{}$ is
given by~\eqref{eqn:ehpsi} and $G$ is given by~\eqref{eqn:Gdef}. The
limit cocycle $j^\psi$ is such that
\[
j^\psi_t( a ) = U_t^* ( a \otimes \id_\mmul ) U_t \qquad %
\text{for all }t \in \R_+ \text{ and } a \in \vna,
\]
where the adapted $\ini$~process $( U_t )_{t \in \R_+}$ is unitary for
all $t \in \R_+$ and satisfies the quantum stochastic differential
equation~\eqref{eqn:rightqsde}; in particular, the map $j^\psi_t$ is a
normal $*$-homomorphism for all $t \ge 0$.
\end{theorem}
\begin{proof}
Fix $a \in \vna$ and let $m$, $G$, $A$, $B$, $C$, $c$, $\tau_0$ and
$r_\tau$ be as in the proof of Theorem~\ref{thm:exhp}, so that, in
particular,
\[
\nmodf{\Phi( \tau )}{\tau}( a ) = %
\sum_{j = 1}^\infty \frac{1}{j!} %
m\bigl( ( a \otimes \id_\statesp ) ( -\I G )^j + %
( \I G )^j ( a \otimes \id_\statesp ) \bigr) + %
\sum_{j, k = 1}^\infty \frac{\I^{j - k}}{j! \, k!} %
m( G^j ( a \otimes \id_\statesp ) G^k ).
\]
From the working in that proof, the first series converges to
$( a \otimes \id_\statesp ) F + F^* ( a \otimes \id_\statesp )$ as
$\tau \to 0$ and, considered as a function of $a$, the convergence
holds in the completely bounded sense.

For the double series, note that
\[
A ( a \otimes \id_\statesp ) B A = %
A B ( a \otimes \id_\statesp ) A = 0 \quad \text{and} \quad %
A ( a \otimes \id_\statesp ) C = C ( a \otimes \id_\statesp ) A = 0;
\]
therefore, after some working,
\[
G ( a \otimes \id_\statesp ) G = %
A ( a \otimes \id_\statesp ) A + %
\tau^{1 / 2} ( B ( a \otimes \id_\statesp ) A + %
A ( a \otimes \id_\statesp ) B ) + %
\tau B ( a \otimes \id_\statesp ) B + \tau^{3 / 2} D_{1, 1}
\]
and, if $j$ and $k$ are not both $1$,
\begin{align*}
G^j ( a \otimes \id_\statesp ) G^k & = %
A^j ( a \otimes \id_\statesp ) A^k + %
\tau^{1 / 2} ( B A^{j - 1} ( a \otimes I_\statesp ) A^k + %
A^j ( a \otimes I_\statesp ) A^{k - 1} B ) \\[1ex]
& \quad + \tau ( A^{j - 1} B ( a \otimes I_\statesp ) B A^{k - 1} + %
B A^{j - 1} ( a \otimes I_\statesp ) A^{k - 1} B ) \\[1ex]
 & \quad + \tau \sum_{l = 0}^{j - 2} %
A^l B^2 A^{j - 2 - l} ( a \otimes \id_\statesp ) A^k + %
\tau \sum_{l = 0}^{k - 2} %
A^j ( a \otimes \id_\statesp ) A^l B^2 A^{k - 2 - l} + %
\tau^{3 / 2} D_{j, k},
\end{align*}
where $\| a \mapsto D_{j, k} \|_\cb \le ( 3 c )^{j + k}$ for all
$\tau \in ( 0, \tau_0 )$.

Furthermore, if $j$, $k \ge 1$ then
\begin{align*}
m( A^j ( a \otimes \id_\statesp ) A^k ) & = %
A^j ( a \otimes \id_\statesp ) A^k, \\[1ex]
m\bigl( \tau^{1 / 2} %
B A^{j - 1} ( a \otimes I_\statesp ) A^k \bigr) & = %
B A^{j - 1} ( a \otimes I_\statesp ) A^k \\[1ex]
\text{and} \quad %
m\bigl( \tau^{1 / 2} %
A^j ( a \otimes I_\statesp ) A^{k - 1} B \bigr) & = %
A^j ( a \otimes I_\statesp ) A^{k - 1} B.
\end{align*}
Also,
\[
m( \tau B ( a \otimes \id_\statesp ) B ) = %
\diag( B ( a \otimes \id_\statesp ) B ) + %
\tau^{1 / 2} r_\tau( B ( a \otimes \id_\statesp ) B ),
\]
whereas, if $j$ and $k$ are not both $1$,
\begin{align*}
m( \tau A^{j - 1} B ( a \otimes \id_\statesp ) B A^{k - 1} ) & = %
\tau^{1 / 2} %
r_\tau( A^{j - 1} B ( a \otimes \id_\statesp ) B A^{k - 1} ) \\[1ex]
\text{and} \quad %
m( \tau B A^{j - 1} ( a \otimes \id_\statesp ) A^{k - 1} B \bigr) & %
= \diag( B A^{j - 1} ( a \otimes \id_\statesp ) A^{k - 1} B ).
\end{align*}
Finally, if
\[
S_{j, k} := %
\sum_{l = 0}^{j - 2} A^l B^2 A^{j - 2 - l} %
( a \otimes \id_\statesp ) A^k + %
\sum_{l = 0}^{k - 2} A^j %
( a \otimes \id_\statesp ) A^l B^2 A^{k - 2 - l}
\]
then $m( \tau S_{j, k} ) = \tau^{1 / 2} r_\tau( S_{j, k} )$. Hence
\begin{multline*}
m( G^j ( a \otimes \id_\statesp ) G^k ) = %
A^j ( a \otimes \id_\statesp ) A^k + %
B A^{j - 1} ( a \otimes \id_\statesp ) A^k + %
A^j ( a \otimes \id_\statesp ) A^{k - 1} B \\
+ \diag( B A^{j - 1} ( a \otimes \id_\statesp ) A^{k - 1} B ) + %
\tau^{1 / 2} D_{j, k}',
\end{multline*}
where
\[
D_{j, k}' = %
r_\tau( A^{j - 1} B ( a \otimes \id_\statesp ) B A^{k - 1} ) + %
r_\tau( S_{j, k} ) + m( \tau D_{j, k} ),
\]
for all $j$, $k \ge 1$. Since
\[
\| a \mapsto D_{j, k}' \|_\cb \le %
( 5 + 5 ( j - 1 + k - 1 ) + 6 \cdot 3^{j + k} ) c^{j + k}
\]
for all $\tau \in \tau_0$ and $j$, $k \ge 1$, the result now follows
by an $M$-test argument, the identity~\eqref{eqn:Upsilon} and
Theorems~\ref{thm:hp} and~\ref{thm:eh}; as~$\tau \to 0$, the double
series
$\sum_{j, k = 1}^\infty \I^{j - k} %
m( G^j ( a \otimes \id_\statesp ) G^k ) / ( j! \, k! )$
tends to
\[
\begin{bmatrix}
\diag_0( H_o( a \otimes \id_{\statesp_0} ) H_o + %
L^* e_1^{\I H_\times} ( a \otimes I_{\statesp_0^\perp} ) %
e_1^{-\I H_\times } L ) & %
\I L^* e_1^{\I H_\times} ( a \otimes \id_{\statesp_0^\perp} ) %
( e^{-\I H_\times} - \id_{\ini \otimes \statesp_0^\perp} ) \\[1ex]
-\I ( e^{\I H_\times} - \id_{\ini \otimes \statesp_0^\perp} ) %
( a \otimes \id_{\statesp_0^\perp} ) e_1^{-\I H_\times} L & %
( e^{\I H_\times} - \id_{\ini \otimes \statesp_0^\perp} ) %
( a \otimes \id_{\statesp_0^\perp} ) %
( e^{-\I H_\times} - \id_{\ini \otimes \statesp_0^\perp} )
\end{bmatrix},
\]
where $e_1^{-\I H_\times}$ is an abbreviation for
$\exp_1( -\I H_\times )$ \etc.
\end{proof}

\begin{remark}
Theorem~\ref{thm:eh} is a generalisation of~\cite[Theorem~5]{Blt10b};
see also \cite[Theorem~4.1 and Remark~3]{Gou04} for the vector-state
case. It provides an explicit description of the Lindblad
generator~$\Lind$ for expectation semigroup of the cocycle~$j^\psi$
which arises in the limit: if $a \in \vna$ then
\begin{align*}
\Lind( a ) & := E^\vac \psi( a ) E_\vac \\[1ex]
 & \hphantom{:}= %
\sstate\bigl( \Psi( a ) \bigr) \\[1ex]
 & \hphantom{:}= -\I [ a, \sstate_0( H_d ) ] - %
\hlf \{ a, \sstate_0( H_o^2 ) \} - %
a \, \sstate_0( L^* \exp_2( -\I H_\times ) L ) - %
\sstate_0( L^* \exp_2( \I H_\times ) L ) \, a \\
 & \hspace{8em} + %
\sstate_0( H_o ( a \otimes \id_{\statesp_0} ) H_o ) + %
\sstate_0( L^* \exp_1( \I H_\times ) %
( a \otimes \id_{\statesp_0^\perp} ) %
\exp_1( -\I H_\times ) L ) \\[1ex]
 & \hphantom{:}= %
-\I [ a, \sstate_0( H_d' ) ] - \hlf \{ a, \sstate_0( D^* D ) \} + %
\sstate_0( D^* ( a \otimes \id_{\statesp_0^\perp} ) D ) \\
 & \hspace{18em} - \hlf \{ a, \sstate_0( H_o^2 ) \} + %
\sstate_0( H_o ( a \otimes \id_{\statesp_0} ) H_o ),
\end{align*}
where
\[
H_d' := H_d - \mbox{$\frac{\I}{2}$} \diag_0( L^* \bigl( %
\exp_2( -\I H_\times ) - \exp_2( \I H_\times ) \bigr) L ) %
\quad \text{and} \quad %
D := -\I \exp_1( -\I H_\times ) L.
\]
This goes beyond the results for Gibbs states contained
in~\cite{AtJ07a}.
\end{remark}

\subsection*{Acknowledgements}
This work was begun while the author was an Embark Postdoctoral Fellow
in the School of Mathematical Sciences, University College Cork,
funded by the Irish Research Council for Science, Engineering and
Technology. A significant part of it was completed while a guest of
Professor Rajarama Bhat at the Indian Statistical Institute,
Bangalore, whose hospitality is gratefully acknowledged; this visit
was made as part of the UKIERI research network \emph{Quantum
Probability, Noncommutative Geometry and Quantum Information}. Thanks
are extended to Professor Martin Lindsay for helpful comments on a
previous draft.

\section*{References}%

\end{document}